\documentclass[a4paper]{amsart}
\makeatletter
\usepackage{geometry,aliascnt}
\usepackage[symbol]{footmisc}
\usepackage{microtype}
\usepackage{graphicx}
\usepackage[all]{xy}
\usepackage{tikz}
\usetikzlibrary{shapes.geometric}
\usepackage[utf8x]{inputenc} 
\usepackage{enumerate,caption}
\usepackage{enumitem}
\usepackage{algorithm}
\usepackage{algpseudocode}
\algrenewcommand\algorithmicrequire{\textbf{Input:}}
\algrenewcommand\algorithmicensure{\textbf{Output:}}
\usepackage{xspace}
\usepackage[modulo]{lineno}
\usepackage[linktocpage,colorlinks=true,linkcolor=blue,citecolor=blue,urlcolor=blue,citebordercolor={0
  0 1},urlbordercolor={0 0 1},linkbordercolor={0 0 1}]{hyperref} 
\usepackage[nameinlink]{cleveref}
\newgeometry{margin=1.7in}
\crefname{diagram}{Diagram}{Diagram}
\usepackage[normalem]{ulem}
\usepackage{soul}
\numberwithin{equation}{section}
\usepackage{amssymb,latexsym,amsmath,amsthm,mathrsfs}
\usepackage{tikz-cd,color}
\theoremstyle{plain}
\newtheorem{theorem}{Theorem}[section]
\newtheorem{corollary}[theorem]{Corollary}
\newtheorem{lemma}[theorem]{Lemma}

\newtheorem{proposition}[theorem]{Proposition}

\newtheorem*{thm*}{Theorem}
\newtheorem{thmx}{Theorem}

\theoremstyle{definition}

\newtheorem{remark}[theorem]{Remark}

\newtheorem{definition}[theorem]{Definition}

\newtheorem{example}[theorem]{Example}

\newenvironment{notation}[1][Notation]
{\begin{trivlist} \item[\hskip \labelsep {\bfseries #1}]}
{\end{trivlist}}
\title[]{A deterministic algorithm for Harder-Narasimhan filtrations for representations of acyclic quivers}
\author{chi-yu Cheng}
\date{}

\DeclareMathOperator{\Proj}{Proj}

\DeclareMathOperator{\Hom}{Hom}

\DeclareMathOperator{\Rep}{Rep}
\DeclareMathOperator{\GL}{GL}
\DeclareMathOperator{\Endo}{End}

\DeclareMathOperator{\Sp}{Sp}
\DeclareMathOperator{\Span}{Span}
\DeclareMathOperator{\disc}{disc}

\begin{document}
\begin{abstract}
    Let $M$ be a representation of an acyclic quiver $Q$ over an infinite field $k$. We establish a deterministic algorithm for computing the Harder-Narasimhan filtration of $M$. The algorithm is polynomial in the dimensions of $M$, the weights that induce the Harder-Narasimhan filtration of $M$, and the number of paths in $Q$. As a direct application, we also show that when $k$ is algebraically closed and when $M$ is unstable, the same algorithm produces Kempf's maximally destabilizing one parameter subgroups for $M$.
\end{abstract}
\maketitle
\tableofcontents
\section{Introduction}

Our goal in this paper is to provide a constructive approach to Harder-Narasimhan filtrations of representations of acyclic quivers. Specifically, we first establish in \Cref{HN_contains_Witness_to_disc_intro} a link between the Harder-Narasimhan filtration and the \emph{discrepancy} of a representation. The later notion, introduced in \cite{MR4238986}, allows us to apply algebraic complexity tools to the study of Harder-Narasimhan filtrations of representations of acyclic quivers. The main result, Theorem 2 of \cite{MR4238986} is a deterministic, polynomial time algorithm for finding witnesses to the discrepancies of representations of bipartite quivers. Huszar then extended the result to acyclic quivers in \cite{https://doi.org/10.48550/arxiv.2111.00039} in Proposition 3.5 (\Cref{discrepancy_rep_algorithm} in our paper). Our first main result \Cref{HN_contains_Witness_to_disc_intro} affirms that for any representation, its Harder-Narasimhan filtration has a term that witnesses its discrepancy. So the next natural question to ask is that is there also a systematic way to compute Harder-Narasimhan filtrations? Combining \Cref{discrepancy_rep_algorithm} and \Cref{HN_contains_Witness_to_disc_intro}, we are able to establish a deterministic algorithm, \Cref{HN_alg_psuedo_code} in \Cref{Effective_alg_intro} that computes the Harder-Narasimhan filtration of a representation of an acyclic quiver.  We would like to point out that \Cref{HN_alg_psuedo_code} is not polynomial in the number of edges of the quiver, but in the number of paths (see \Cref{thmB_part_I}). When the quiver is bipartite where arrows only go from one partite to the other, the number of arrows equals that of the paths. In this case \Cref{HN_alg_psuedo_code} indeed has polynomial time complexity.  

This paper was originally motivated by the relations between "maximally destabilizing subobjects" under different stability conditions in the case of representations of quivers. Stability plays an important role in algebraic geometry for constructing moduli spaces of algebro-geometric objects. In the case of representations of quivers, two commonly used stability conditions are weight stability and slope stability. We now briefly introduce the two stability conditions, and we shall see that failing either one of them would require some subobject that contradicts the condition. We let $Q$ be an acyclic quiver.\footnote[2]{The assumption that the quiver is acyclic is not required to define either stability condition. Nevertheless, we will stick to acyclic quivers throughout the paper for consistency, as the algorithm deals with acyclic quivers only.}  That is, $Q$ does not have an oriented cycle, and we let $Q_0$ denote the set of vertices of $Q$. 

Weight stability was formulated in \cite{MR1315461} to construct moduli of representations of finite dimensional algebras. King considered representations of $Q$ of a fixed dimension $\mathbf{d}\in\mathbf{N}^{Q_0}$, and a $\mathbf{Z}$-linear weight function $\theta:\mathbf{Z}^{Q_{0}}\rightarrow\mathbf{Z}$ such that $\theta(\mathbf{d})=0$. For a representation $M$ of $Q$, we let $\pmb{\dim}M\in \mathbf{N}^{Q_{0}}$ be its dimension vector. By $\theta(M)$ we mean $\theta(\pmb{\dim}M).$ King defined that for a representation $M$ of dimension $\mathbf{d}$, $M$ is $\theta$-\emph{semistable} if and only if  $\theta(M^{\prime})\leq \theta(M)=0$ for every subrepresentation $M^{\prime}$ of $M$. Otherwise, $M$ is $\theta$-\emph{unstable}. Namely, there is a subrepresntation $M^{\prime}$ of $M$ such that $\theta(M^{\prime})>0.$ In \cite{MR4238986}, the \emph{discrepancy} of any representation $M$ \emph{with respect to }$\theta$ is defined to be the number  
$$\disc(M,\theta) = \max _{M^{\prime}\subseteq M}\theta(M^{\prime}),$$ where $M^{\prime}\subseteq M$ means $M'$ is a subrepresentation of $M$. We see that if $M$ is $\theta$-unstable, then the subrepresentation that witnesses  $\disc(M,\theta)$ contradicts $\theta$-semistability of $M$ the most.

On the other hand, weight stability is in fact a reinterpretation of the Hilbert-Mumford criterion (\Cref{HilbMumAffine}).  When $M$ is $\theta$-unstable, the original Hilbert-Mumford criterion states that there is a one parameter subgroup (of the group to be defined in \Cref{the_quiver_setting}) that contradicts $\theta$-semistability for $M$. Kempf showed in \cite{MR506989} that among all destabilizing one parameter subgroups, there is a unique indivisible one, up to conjugation by  a parabolic subgroup, that maximally contradicts $\theta$-semistability in the numerical sense that will be made precise in  \Cref{from_HN_to_OPS}. The maximally destabilizing one parameter subgroups of an unstable representation induce a unique filtration, which was referred to as Kempf filtration in \cite{MR3199484}. The other type of filtration, known as the Harder-Narasimhan filtration, comes from slope instability.

Slope stability depends on two weights $\Theta,\kappa:\mathbf{Z}^{Q_0}\rightarrow\mathbf{Z}$, where $\kappa((\mathbf{Z}^{+})^{Q_0})>0$. The slope $\mu$ of a nonzero representation $M$ is defined as $$\mu(M)=\frac{\Theta(M)}{\kappa(M)}.$$ Unlike weight stability, slope stability is defined for any nonzero representation, whether its slope is zero or not. In \cite{MR1906875}, a nonzero representation $M$ is said to be $\mu$-semistable if $\mu(M')\leq \mu(M)\text{ for all }0\neq M'\subseteq M.$ In \Cref{stability_conditions_for_representations_of_quivers}, we recall in \Cref{connecting_slope_weight_stability} a  way to go between weight and slope stability. In any case,  we shall stick to the following convention throughout the paper, whose convenience will become apparent later: If a statement only involves weight stability, we shall use the lower case $\theta$ as the weight. When slope stability is involved, we use the upper case $\Theta$ as the numerator of the slope. 

The Harder-Narasimhan filtration of a representation $M$ with respect to $\mu$ is constructed inductively in \cite{MR1906875}. As the first step, set $M_1$ as the unique subrepresentation that is maximal among all subrepresentations having the highest slope. Such a subrepresentation is called the \emph{strongly contradicting semistability }(abbreviated as scss) \emph{subrepresentation of }$M$. One then goes on to set $M_{i+1}/M_{i}$ as the scss subrepresentation of $M/M_i$ until $M/M_i$ is $\mu$-semistable. The resulting filtration $0\subsetneq M_1\subsetneq \cdots\subsetneq M_r=M$ is called the \emph{Harder-Narasimhan filtration} of $M$.

We now have the following maximally destabilizing subobjects associated to an unstable representation $M$:
\begin{enumerate}
    \item The discrepancy of $M$, and its witnessing subrepresentations (not necessarily unique).
    \item The scss subrepresentation of $M$, and the Harder-Narasimhan filtration of $M$.
    \item The one parameter subgroups that maximally contradict the semistability of $M$, and the Kempf filtration induced by those one parameter subgroups. 
\end{enumerate}

The relation between (2) and (3) is already unraveled by the paper \cite{MR3199484}. The main result is that for an unstable representation, its Harder-Narasimhan filtration coincides with its Kempf filtration. The first of our two main results \Cref{HN_contains_Witness_to_disc_intro} establishes a link between (1) and (2). The second main result \Cref{Effective_alg_intro} establishes a deterministic algorithm (\Cref{HN_alg_psuedo_code}) for computing the Harder-Narasimhan filtration and therefore the  maximally destabilizing one parameter subgroups of any unstable representation. 

\subsection{The main results}
We let $Q$ be an acyclic quiver. 
\begin{thmx}[\Cref{discrep_contains_1st_HN_term}, \Cref{HN_contains_witness_discrepancy},  \Cref{HN_contains_witness_discrepancy_corollary}]\label{HN_contains_Witness_to_disc_intro}
Let $\Theta$ be a weight and let $M$ be a representation of $Q$ with $\Theta(M) = 0$. For any weight $\kappa$ with $\kappa((\mathbf{Z}^{+})^{Q_0})>0$, if $\mu=\Theta/\kappa$ is the slope function, then any subrepresentation $M'$ with $\Theta(M')=\disc(M,\Theta)$ contains the scss subrepresentation of $M$ (with respect to $\mu$). Moreover, if $$0\subsetneq M_{1}\subsetneq \cdots \subsetneq M_{r-1}\subsetneq M_{r}=M$$ is the Harder-Narasimhan filtration of $M$ (with respect to $\mu$), then there is an $M_{l}$ in the filtration such that $\Theta(M_l)=\disc(M,\Theta)$. 
\end{thmx}

It is noteworthy that the discrepancy depends only on $\Theta$, but the Harder-Narasimhan filtration depends on $\mu$, which has an extra piece $\kappa$. While the Harder-Narasimhan filtration of $M$ may change due to different choices of $\kappa$, the results of \Cref{HN_contains_Witness_to_disc_intro} are independent of $\kappa$ for a fixed $\Theta$. 

\begin{thmx}[\Cref{thmB_part_I}, \Cref{one_PS_via_Kempf's_filtration}]\label{Effective_alg_intro}
Let $M$ be a representation of $Q$ over an infinite field. There exists a deterministic algorithm to compute the Harder-Narasimhan filtration of $M$. When the ground field is algebraically closed and when $M$ is unstable, a maximally destabilizing one parameter subgroup for $M$ can be constructed based on its Harder-Narasimhan filtration. Hence the algorithm also produces a maximally destabilizing one parameter subgroup for $M$. Moreover, in the case that $Q$ is bipartite where all arrows go from one partite to the other, the algorithm has polynomial time complexity.
\end{thmx}

\Cref{Effective_alg_intro} is inviting for several purposes. First, finding Kempf's maximally destabilizing one parameter subgroups in general is very hard. Roughly speaking, suppose $V$ is a representation of a reductive group $G$. If $G$ is a torus, then finding Kempf's one parameter subgroup for an unstable point $v\in V$ is done by linear programming, whose constraints come from the states of $v$. For a general reductive group, one can first restrict the action to a maximal torus, then conduct linear programming on the entire orbit $G\cdot v$. The set of states of all points in the orbit can be as large as the power set of $\{1,2,\ldots,\dim V\}$, minus the empty subset. Hence the time complexity to compute Kempf's one parameter subgroups using brute force is exponential. \Cref{Effective_alg_intro} provides an alternative to brute force for finding Kempf's one parameter subgroups in the context of representations of acyclic quivers. In addition, if the quiver is bipartite as stated in \Cref{Effective_alg_intro}, the algorithm has polynomial time complexity. 

Second, \cite{MR3261979} shows that the stratification of the space of representations of a quiver of a fixed dimension by Harder-Narasimhan types coincides with the stratification by Kempf's one parameter subgroups. Hence  \Cref{Effective_alg_intro} sheds new light on computing the stratification by either type. 

\subsection{Outline of the paper}
We briefly recall slope stability (\Cref{def_slope_stab}), weight stability (\Cref{weight_stability_def}), and a way to go between the two (\Cref{connecting_slope_weight_stability}) in \Cref{stability_conditions_for_representations_of_quivers}. In \Cref{the_main_results_main_body}, we establish \Cref{HN_contains_Witness_to_disc_intro} (\Cref{discrep_contains_1st_HN_term}, \Cref{HN_contains_witness_discrepancy}, \Cref{HN_contains_witness_discrepancy_corollary}) and the first part of \Cref{Effective_alg_intro} (\Cref{thmB_part_I}). Specifically the first part of \Cref{Effective_alg_intro} constructs \Cref{HN_alg_psuedo_code}, and calculates its complexity. The second part of \Cref{Effective_alg_intro} constructs maximally destabilizing one parameter subgroups for unstable representations from their Harder-Narasimhan filtrations. Due to the technicality involved in invariant theory, we postpone the second part of \Cref{Effective_alg_intro} to \Cref{from_HN_to_OPS}. 

In \Cref{affine_GIT}, we recall the Hilbert-Mumford criterion (\Cref{HilbMumAffine}), and define maximally destabilizing one parameter subgroups (\Cref{worst_one_PS_def}). We also present Kempf's main result, \Cref{theorem6} from the paper \cite{MR506989}. \Cref{theorem6} describes the existence and the uniqueness of maximally destabilizing one parameter subgroups. We then apply the machinery to the case of representation of quivers at \Cref{the_quiver_setting}. Finally, we are able to finish the second part of \Cref{Effective_alg_intro} at \Cref{one_PS_via_Kempf's_filtration}.

In \Cref{a_short_example}, we construct an unstable representation to showcase that not all subrepresentations that witness the discrepancy occur in the Harder-Narasimhan filtration. We also provide two approaches to compute its maximally destabilizing one parameter subgroups. The first one uses \Cref{Effective_alg_intro} and the second one is the brute force procedure that can be applied to any representation of any reductive group. Although the second approach works in more generality, the cost will be apparent in the example. The point is to give the readers an idea of the hardness of finding maximally destabilizing one parameter subgroups and the value of \Cref{Effective_alg_intro}.

We note that our \Cref{HN_alg_psuedo_code} relies on \Cref{discrepancy_rep_algorithm}, which is an algorithm that computes  discrepancies and was established in \cite{https://doi.org/10.48550/arxiv.2111.00039}. \Cref{discrepancy_rep_algorithm} in turn is a consequence of \Cref{thm_1.5_ISQ18} from \cite{MR3868734}. We supply an almost self-contained account of how \Cref{discrepancy_rep_algorithm} is established based on \Cref{thm_1.5_ISQ18} at \Cref{appendixB}. Although this was the work of \cite{https://doi.org/10.48550/arxiv.2111.00039}, we intend to fill in some details not supplied in the original paper. Here is a visualization of the logical dependency of the results in our paper:
$$
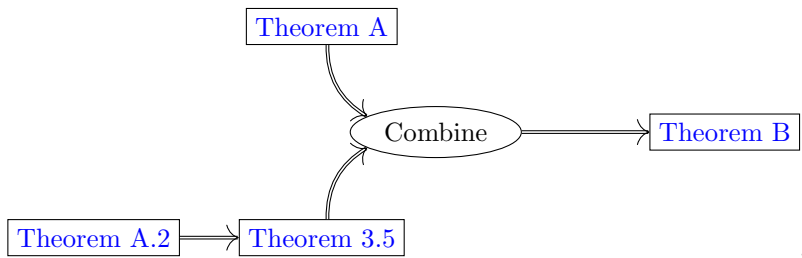
\begin{tikzpicture}
\node (a) at (-6,-1.4) [rectangle,draw] {\Cref{thm_1.5_ISQ18}};
\node (b) at (-3,-1.4) [rectangle,draw] {\Cref{discrepancy_rep_algorithm}};
\node (c) at (2.3,0) [rectangle,draw] {\Cref{Effective_alg_intro}};
\node (d) at (-3,1.4) [rectangle,draw] {\Cref{HN_contains_Witness_to_disc_intro}};
\path (-1.5,0) node (e) [ellipse, draw]{Combine};
\draw [->, double] (a) to (b);
\draw [->,double,bend left] (b) to (e);
\draw [->,double,bend right] (d) to (e);
\draw [->,double] (e) to (c);
\end{tikzpicture}.$$

Finally, we will work with an arbitrary ground field throughout \Cref{stability_conditions_for_representations_of_quivers} and \Cref{the_main_results_main_body} until \Cref{discrepancy_rep_algorithm}. Originally, the machinery of \Cref{discrepancy_rep_algorithm} only requires the ground field to have \emph{enough} elements. For simplicity we will work with an infinite field, starting from \Cref{discrepancy_rep_algorithm}. In \Cref{from_HN_to_OPS}, we will assume further that the ground field is algebraically closed for the formulation of Kempf's maximally destabilizing one parameter subgroups.
\subsection{Acknowledgements}
The author would like to thank Calin Chindris for his support along the preparation of this paper. The discussion with him and questions raised by him offered many insights that fostered the results in the paper. The author also gives thanks to Alfonso Zamora, Alana Huszar, and G\'{a}bor Ivanyos, Youming Qiao for kindly answering the author's questions on certain results in the paper \cite{MR3199484}, \cite{https://doi.org/10.48550/arxiv.2111.00039}, and \cite{MR3868734} respectively. Finally, the author thanks a referee for clarifying the complexity of the algorithms presented.

\section{Stability conditions for representations of quivers}\label{stability_conditions_for_representations_of_quivers}
In this section we recall slope stability (\Cref{def_slope_stab}) and weight stability (\Cref{weight_stability_def}) for representations of quivers.

Weight stability defined in this section is originally the Hilbert-Mumford criterion (\Cref{HilbMumAffine}) formulated in \cite{MR1315461} in the context of representations of quivers. For the purpose of the exposition of this paper, we postpone a more detailed account of weight stability to \Cref{from_HN_to_OPS}. At the end of this section we recall in \Cref{connecting_slope_weight_stability} a way to translate between slope stability and weight stability.
\subsection{Set up}
Let $k$ be a field, not necessarily algebraically closed. A quiver $Q=(Q_0,Q_1,t,h)$ consists of $Q_0$ (vertices) and $Q_1$ (arrows) together with two maps $t,h:Q_1\rightarrow Q_0$. The image of an arrow under $t$ (resp. $h$) is the tail (resp. head) of the arrow. We represent $Q$ as a directed graph with vertices $Q_0$ and directed arrows $a:ta\rightarrow ha$ for $a\in Q_1$. A \emph{path} in $Q$ is a composition of arrows in $Q$. We give each vertex $v$ the trivial arrow $e_v:v\rightarrow v$. For any path $p:x\rightarrow v$ (resp. $q:v\rightarrow y$), we set $e_vp=p$ (resp. $qe_v=q$). We let $Q$ be a finite acyclic quiver throughout the paper. That is, both $Q_0$ and $Q_1$ are finite, and $Q$ does not have any oriented cycle other than the self-loops $\{e_v\}_{v\in Q_0}$. In this way, the number of paths in $Q$ is finite.

A representation $M$ of $Q$ consists of a collection of finite dimensional $k$-vector spaces $M_{v}$, for each $v\in Q_0$ together with a collection of $k$-linear maps $M(a):M_{ta}\rightarrow M_{ha}$, for each $a\in Q_1$. For each $v\in Q_0$, we set  $M(e_v):M_v\rightarrow M_v$ to be the identity map. 

If $M,N$ are two representations of $Q$, then a morphism from $M$ to $N$ is a collection of $k$-linear maps $\varphi_v:M_v\rightarrow N_v$ for each $v\in Q_0$, such that $N(a)\varphi_{ta}=\varphi_{ha}M(a)$ for each $a\in Q_1$. In particular, $M'$ is a subrepresentation of $M$ whenever each $M'_v$ is a subspace of $M_v$ and each $M'(a)$ is the restriction of $M(a)$ to $M'_{ta}$. We will write $M'\subseteq M$ if $M'$ is a subrepresentation of $M$. We also note that the collection of representations of $Q$ forms an abelian category.

The dimension vector of a representation $M$ of $Q$ is the tuple $\pmb{\dim} M\in \mathbf{Z}^{Q_0}$ where $(\pmb{\dim} M)_v = \dim M_v$ as a $k$-vector space for each $v\in Q_0$. A \emph{weight} is a $\mathbf{Z}$-linear function $\mathbf{Z}^{Q_0}\rightarrow\mathbf{Z}$. If $\theta$ is a weight, we write $\theta(\pmb{\dim} M)$ as $\theta(M)$.

For clarity, below is a summary of some important notations that we stick to throughout the paper:
\begin{notation}
Let $M$ be a representation of $Q$. For each vertex $v\in Q_0$, we write $M_v$ as the vector space of $M$ at $v$. For each $a\in Q_1$, $M(a)$ denotes the map $M_{ta}\rightarrow M_{ha}$ that is part of the data of $M$. For any weight  $\theta:\mathbf{Z}^{Q_0}\rightarrow\mathbf{Z}$, $\theta(v)$ instead of $\theta_v$ is the weight of $\theta$ at $v$. We also decree that $\theta(M)$ means $\sum_{v\in Q_0}\theta(v)\dim M_v$. For a dimension vector $\mathbf{d}\in \mathbf{Z}^{Q_0}$, we write $\mathbf{d}_v$ instead of $\mathbf{d}(v)$ for its component at $v$. 
\end{notation}

\subsection{Slope stability}
In this section we recall semistability of representations of quivers with respect to a slope, along with some important properties stated in \Cref{scissor} and \Cref{scss_lem}. We then recall the construction from \cite{MR1906875} of the Harder-Narasimhan filtration of a representation at \Cref{HN}. These results will be widely used at \Cref{the_main_results_main_body}. 

Fix two weights $\Theta,\kappa:\mathbf{Z}^{Q_0}\rightarrow\mathbf{Z}$ where we require $\kappa((\mathbf{Z}^{+})^{Q_0})> 0$.  We set $\mu(M)=\Theta(M)/\kappa(M)$ for any nonzero representation $M$. We implicitly assume that a representation is nonzero whenever the slope function is applied to it. 
\begin{definition}\label{def_slope_stab}
We say a representation $M$ of $Q$ is $\mu$-\emph{semistable} if $\mu(N)\leq \mu(M)$ for every subrepresentation $N$ of $M$. 
\end{definition}

\begin{lemma}[Lemma 2.1 \cite{MR1906875} ]\label{scissor}
Let $0\rightarrow L\rightarrow M\rightarrow N\rightarrow 0$ be a short exact sequence of representations of $Q$. Then the following conditions are equivalent:
\begin{enumerate}
    \item $\mu(L)\leq \mu(M)$,
    \item $\mu(L)\leq\mu(N)$,
    \item $\mu(M)\leq\mu(N)$.
\end{enumerate}
\end{lemma}

A fundamental result is the following: 
\begin{lemma}[Lemma 2.2 \cite{MR1906875}]\label{scss_lem}
Let $M$ be a representation of $Q$. There is a unique subrepresentation $N$ such that 
\begin{enumerate}
    \item $\mu(N)$ is maximal among subrepresentations of $M$, and 
    \item if $\mu(N')=\mu(N)$, then $N'\subseteq N$.
\end{enumerate}
\end{lemma}

We call such a representation $N$ the \emph{strongly contradicting semistability }(abbreviated as scss) \emph{subrepresentation of }$M$.

\begin{theorem}[Theorem 2.5 \cite{MR1906875}]\label{HN}
Let $M$ be a representation of $Q$. There is a unique filtration $0=M_0\subsetneq M_1\subsetneq \cdots\subsetneq M_{r-1}\subsetneq M_r = M$ such that 
\begin{enumerate}
\item $\mu(M_i/M_{i-1})>\mu(M_{i+1}/M_{i})$ for $i=1,\ldots, r-1$,
\item the representation $M_i/M_{i-1}$ is $\mu$-semistable for $i=1,\ldots,r$.
\end{enumerate}
\end{theorem}

The above filtration is called the \emph{Harder-Narasimhan filtration of} $M$. It is constructed by setting $M_1$ to be the scss subrepresentation of $M$, then by inductively setting $M_{i+1}/M_{i}$ as the scss subrepresentation of $M/M_{i}$ until $M/M_i$ is $\mu$-semistable.

\subsection{Weight stability}
Another stability condition on representations of quivers comes from GIT and was introduced in \cite{MR1315461}. Like the slope stability we introduced earlier, GIT stability also depends on a choice of parameters. In the case of representations of quivers the parameter is a single weight.  We therefore refer to this stability condition as weight stability. We will relate weight and slope stability in \Cref{connecting_slope_weight_stability} at the end of this section. 

\begin{definition}\label{weight_stability_def}
Let $\mathbf{d}\in\mathbf{Z}^{Q_0}$ be a dimension vector and let $\theta$ be a weight with $\theta(\mathbf{d})=0$. A representation $M$ with dimension vector $\mathbf{d}$ is $\theta$-\emph{semistable} if $\theta(N)\leq 0$ for all subrepresentations $N\subseteq M$. 
\end{definition}

The notion of \emph{discrepancy} was introduced in \cite{MR4238986} to describe the subrepresentations that contradict the weight stability the most. This very notion allows us to compute the  Harder-Narasimhan filtrations of representations of an acyclic quiver using tools from algebraic complexity, as we shall see in \Cref{the_main_results_main_body}.

\begin{definition}\label{discrepancy}
Let $M$ be a representation and let $\theta$ be a weight. We do not require $\theta(M)=0$. 
The \emph{discrepancy of }$M$ \emph{with respect to} $\theta$ is defined as 
$$\disc(M,\theta)=\max_{N\subseteq M}\theta(N).$$ We say a subrepresentation $N$ is $\theta$-\emph{optimal in }$M$, or $N$ \emph{witnesses} $\disc(M,\theta)$ if $\theta(N)=\disc(M,\theta)$.
\end{definition}

We recall the following way to go between weight stability and slope stability:
\begin{lemma}\label{connecting_slope_weight_stability}
Let $\Theta$ and $\kappa$ be two weights with $\kappa(N)> 0$ for every nonzero representation $N$. Let $\mu$ be the corresponding slope $\Theta/\kappa$. For every dimension vector $\mathbf{d}\in\mathbf{Z}^{Q_0}$, define the new weight $$\theta_{\mathbf{d}}(N): =\kappa(\mathbf{d})\Theta(N)-\Theta(\mathbf{d})\kappa(N)\text{ for each }N.$$ Then a representation $M$ of dimension $\mathbf{d}$ is $\mu$-semistable if and only if $M$ is $\theta_{\mathbf{d}}$-semistable. Moreover, the new slope $\mu_{\mathbf{d}}=\theta_{\mathbf{d}}/\kappa$ defines the same slope stability condition so that the Harder-Narasimhan filtrations of a representation with respect to $\mu$ and $\mu_{\mathbf{d}}$ coincide.
\end{lemma}

\begin{proof}
Let $M$ be a representation of dimension $\mathbf{d}$ and let $N\subset M$ be any subrepresentation. We then have 
\begin{equation*}
\begin{split}
    \mu(N)=\frac{\Theta(N)}{\kappa(N)}\leq \mu(M)=\frac{\Theta(\mathbf{d})}{\kappa(\mathbf{d})}\Leftrightarrow \kappa(\mathbf{d})\Theta(N)-\Theta(\mathbf{d})\kappa(N)=\theta_{\mathbf{d}}(N)\leq 0
    \end{split}
\end{equation*}
The first statement follows.

For the second statement, simply note that $$\mu_{\mathbf{d}}(N) = \kappa(\mathbf{d})\cdot \mu(N)-\Theta(\mathbf{d}).$$ Hence $\mu_{\mathbf{d}}$ is a positive scalar multiple of $\mu$, followed by a translation.
\end{proof}

\section{The main results}\label{the_main_results_main_body}
Fix an acyclic quiver $Q$. In this section we prove two major results: \Cref{HN_contains_witness_discrepancy} and \Cref{thmB_part_I}. \Cref{HN_contains_witness_discrepancy} provides a sufficient condition for the Harder-Narasimhan filtration of a representation of $Q$ to contain a term that witnesses the discrepancy. We then propose  \Cref{HN_alg_psuedo_code} to compute the Harder-Narasimhan filtrations of representations of $Q$. The correctness of \Cref{HN_alg_psuedo_code} is ensured by \Cref{Algorithm_produces_M1}. We also recall the algebraic complexity machinery  \Cref{discrepancy_rep_algorithm} derived from \cite{MR3868734}.  With the aid of \Cref{discrepancy_rep_algorithm}, we establish in \Cref{thmB_part_I} the time complexity of \Cref{HN_alg_psuedo_code}. The derivation of \Cref{discrepancy_rep_algorithm} is supplied in \Cref{appendixB} for interested readers.

We fix two weights $\Theta,\kappa$ on $\mathbf{Z}^{Q_0}$ where $\kappa(N)>0$ for each non-zero representation $N.$ We let $\mu=\Theta/\kappa$ be the slope. We do not assume anything about the ground field until \Cref{discrepancy_rep_algorithm}, where we will begin to assume the ground field is infinite. 

We now present \Cref{discrep_contains_1st_HN_term}, which is the cornerstone of this paper. 

\begin{lemma}\label{discrep_contains_1st_HN_term}
Let $M$ be a $\mu$-unstable representation, and let $M_1$ be the scss subrepresentation of $M$. Suppose $\Theta(M_1)>0$ and $M^{\prime}$ is $\Theta$-optimal in $M$. Then $M^{\prime}$ contains $M_{1}$. Moreover, $M^{\prime}$ is $\mu$-semistable if and only if $M^{\prime}=M_1$.
\end{lemma}

\begin{proof}
Note that $\Theta(M_1)>0$ implies $\disc(M,\Theta)=\Theta(M^{\prime})>0$. 
Suppose on the contrary that $M_{1}\not\subset M^{\prime}$. We then have a proper inclusion $M^{\prime}\subsetneq M_{1}+M^{\prime}$. In particular, $\kappa(M^{\prime})<\kappa(M_{1}+M^{\prime})$. On the other hand, $\Theta(M^{\prime})\geq \Theta(M_{1}+M^{\prime})$ as $M^{\prime}$ attains the discrepancy. If $\Theta(M_1+M^{\prime})\leq 0$, we automatically have $\mu( M^{\prime})>\mu(M_1+M^{\prime}).$ If $\Theta(M_1+M^{\prime})>0$, we have 
$$\frac{\Theta(M^{\prime})}{\kappa(M^{\prime})}\geq \frac{\Theta(M_1+M^{\prime})}{\kappa(M^{\prime})}>\frac{\Theta(M_1+M^{\prime})}{\kappa(M_1+M^{\prime})}.$$ In any case we deduce that $\mu(M^{\prime})>\mu(M_{1}+M^{\prime})$. By \Cref{scissor}, the following short exact sequence 
$$
0\rightarrow M^{\prime}\rightarrow M_{1}+M^{\prime}\rightarrow (M_{1}+M^{\prime})/M^{\prime}\simeq M_{1}/(M_{1}\cap M^{\prime})\rightarrow 0$$ implies $\mu(M^{\prime})>\mu(M_{1}+M^{\prime})>\mu\big(M_{1}/(M_{1}\cap M^{\prime})\big)$.

On the other hand, since $M_1$ is the scss subrepresentation of $M$, we have $\mu(M_{1}\cap M^{\prime})\leq \mu(M_1)$. Now \Cref{scissor} applied to the following short exact sequence 
$$0\rightarrow M_{1}\cap M^{\prime}\rightarrow M_{1}\rightarrow M_{1}/(M_{1}\cap M^{\prime})\rightarrow 0$$ implies 
$\mu\big(M_{1}/(M_{1}\cap M^{\prime})\big)\geq \mu(M_{1})$. In sum, we obtained 
$$\mu(M^{\prime})>\mu(M_{1}+M^{\prime})>\mu\big(M_{1}/(M_{1}\cap M^{\prime})\big)\geq \mu(M_{1}).$$ However, $M_{1}$ is the  scss subrepresentation of $M$. We arrive at a contradiction.

For the second implication of the lemma, if $M^{\prime}$ is the scss subrepresentation of $M$, it is obviously $\mu$-semistable. Conversely, if $M^{\prime}$ is $\mu$-semistable, since it contains $M_1$, we must have $\mu(M_1)\leq \mu(M^{\prime})$. This implies $\mu(M_1)=\mu(M^{\prime})$ so that $M^{\prime} = M_1$.
\end{proof}

Now we can make sense of the following proposition.

\begin{proposition}\label{dsicrepancy_preserved_quotient_by_M_1}
Let $M$ be a $\mu$-unstable representation, and let  $M_1$ be the scss subrepresentation of $M$ with $\Theta(M_1)>0$. If $M^{\prime}$ is $\Theta$-optimal in $M$, we then have $$\Theta(M^{\prime}/M_1)=\disc(M/M_1,\Theta).$$ Namely, $M^{\prime}/M_1$ is $\Theta$-optimal in $M/M_1$.
\end{proposition}

\begin{proof}
Suppose $N/M_1$ is $\Theta$-optimal in $M/M_1$. 
We then obviously have $$\Theta(M^{\prime}/M_1)\leq \disc(M/M_1,\Theta)=\Theta(N/M_1).$$ This is equivalent to $\Theta(M')\leq \Theta(N)$. However, $M^{\prime}$ is $\Theta$-optimal in $M$ by assumption. So we must have $\Theta(M')=\Theta(N)$, resulting in  $\Theta(M^{\prime}/M_1)= \Theta(N/M_1)=\disc(M/M_1,\Theta)$. The proposition is proved.
\end{proof}

\begin{theorem}\label{HN_contains_witness_discrepancy}
Let $M$ be a $\mu$-unstable representation and let $$0\subsetneq M_{1}\subsetneq \cdots \subsetneq M_{r-1}\subsetneq M_{r}=M$$ be its Harder-Narasimhan filtration. Suppose there is an integer $l$ such that $\mu(M_{l}/M_{l-1})>0$ but $\mu(M_{l+1}/M_{l})\leq 0$. Then the term $M_l$ in the filtration is $\Theta$-optimal in $M.$ In addition, $M_l$ has the highest slope among all $\Theta$-optimal subrepresentations.
\end{theorem}

\begin{proof}
We will prove the theorem by induction on $l$. Let $l=1$. We need to show that $\Theta(M_{1})\geq \Theta(M^{\prime})$ for every $M^{\prime}\subset M$. We break this down into several steps.
\begin{enumerate}
    \item For any $M^{\prime}_{1}\subset M_{1}$, we show that $\Theta(M_{1}^{\prime})\leq \Theta(M_{1})$. 
    \item Next, assuming on the contrary that there is an $M^{\prime}\subset M$ with $\Theta(M^{\prime})>\Theta(M_{1})$, we prove that such an $M^{\prime}$ induces the proper inclusion $M_{1}\subsetneq M_{1}+M^{\prime}$.
    \item  Finally, we deduce from step (1) and (2) that $\mu\big((M_{1}+M^{\prime})/M_{1})>\mu(M_2/M_{1}\big)$, contradicting the fact that $M_2/M_{1}$ is the scss subrepresentation of $M/M_1$.
\end{enumerate}
For step (1), suppose on the contrary that there is an $M_1^{\prime}\subsetneq M_1$ such that $\Theta(M_1^{\prime})>\Theta(M_1)$. We would then have $$\mu(M_1^{\prime})=\frac{\Theta(M_1^{\prime})}{\kappa(M_1^{\prime})}>\frac{\Theta(M_1)}{\kappa(M_1^{\prime})}> \frac{\Theta(M_1)}{\kappa(M_1)}=\mu(M_1).$$ Note that again we are using the assumption that $\Theta(M_1)$ is positive. Since $M_1$ is the scss subrepresentation of $M$, we arrive at a contradiction.  Step (1) is completed.

For step (2), since $M_{1}$ is scss in $M$, we get $\mu(M^{\prime})\leq \mu(M_{1})$. This together with the assumption that $\Theta(M^{\prime})>\Theta(M_{1})$ imply $\kappa(M^{\prime})>\kappa(M_{1})$. This clearly implies that $M^{\prime}$ is not contained in $M_{1}$ so that $M_{1}\subsetneq M_{1}+M^{\prime}$, finishing step (2).

For step (3), simply note that $\Theta(M^{\prime})>\Theta(M_{1})\geq\Theta(M_{1}\cap M^{\prime})$ by step (1). We then have  
\begin{equation*}\begin{split}\mu\big((M_{1}+M^{\prime})/M_1\big)&=\mu\big(M^{\prime}/(M_{1}\cap M^{\prime})\big)=\frac{\Theta(M^{\prime})-\Theta(M_{1}\cap M^{\prime})}{\kappa(M^{\prime})-\kappa(M_{1}\cap M^{\prime})}\\
&>0\geq \mu(M_2/M_{1}).\end{split}\end{equation*} These complete all three steps and therefore the base case of the induction.

If $\mu(M_{l+1}/M_{l})>0$ but $\mu(M_{l+2}/M_{l+1})\leq 0$, consider the Harder-Narasimhan filtration for $M/M_{1}$: 
$$0\subsetneq M_{2}/M_{1}\subsetneq \cdots\subsetneq M/M_{1}.$$ By the induction hypothesis, $\Theta(M_{l+1}/M_{1})=\disc(M/M_{1},\theta)$. Moreover, \Cref{dsicrepancy_preserved_quotient_by_M_1} implies $\disc(M/M_1,\Theta)=\disc(M,\Theta)-\Theta(M_1)$. Combining these two, we get $\Theta(M_{l+1})=\disc(M,\Theta)$, as desired. 

Finally, the maximality of slope of $M_l$ can also be proved by induction on $l$, together with \Cref{discrep_contains_1st_HN_term}.
\end{proof}

\begin{corollary}\label{HN_contains_witness_discrepancy_corollary}
Suppose $M$ is a $\mu$-unstable representation of $Q$, and $\Theta(M)=0$. Then there exists a term $M_i$ in the Harder-Narasimhan filtration of $M$ such that $\Theta(M_i)=\disc(M,\Theta)$, and such that $M_i$ has the highest slope among all subrepresentations of $M$ that witness $\disc(M,\Theta)$.
\end{corollary}

\begin{proof}
Let $0\subsetneq M_1\subsetneq\cdots\subsetneq M_r =M$ be the Harder-Narasimhan filtration of $M$. The integer $l$ described in \Cref{HN_contains_witness_discrepancy} must exist. For if this is not the case, we then have $$0=\Theta(M)>\Theta(M_{r-1})>\cdots>\Theta(M_1)>0,$$ which is absurd. 
\end{proof}

We now turn our attention to the computation of Harder-Narasimhan filtrations for representations of $Q$. This comes down to computing the scss subrepresentation. The main ideas come from \Cref{discrep_contains_1st_HN_term}, and the following machinery that works for arbitrary infinite fields:
\begin{theorem}\label{discrepancy_rep_algorithm}
Let $Q$ be an acylcic quiver and let $M$ be a representation of $Q$ over an infinite field. Fix a weight $\theta:\mathbf{Z}^{Q_0}\rightarrow\mathbf{Z}$ with $\theta(M)=0$. There is a deterministic algorithm that finds the discrepancy $\disc(M,\theta)$, together with a subrepresentation $M'$ so that $\theta(M')=\disc(M,\theta)$. If we set  $\Omega=\sum_{v\in Q_0}|\theta(v)|$, $K=\sum_{v\in Q_0}\dim M_v$, and $P$ as the number of paths in $Q$, then the algorithm has run time complexity that is polynomial in $\Omega,K,P$. 
\end{theorem}

We refer interested readers to \Cref{appendixB} for a proof of the theorem. It is presented at \Cref{proposition3.5_hus}. We now explain how \Cref{discrep_contains_1st_HN_term}, together with the algorithm of \Cref{discrepancy_rep_algorithm} can be applied for deriving the scss subrepresentation. For convenience, we adopt the following notations:
\begin{notation}
For any representation $M$ of $Q$ of dimension $\mathbf{d}$, $F(M)$ denotes the output subrepresentation of the algorithm in \Cref{discrepancy_rep_algorithm} applied to $M$ and to the weight $\theta_{\mathbf{d}}$.  We let $G(M)=\disc(M,\theta_{\mathbf{d}})=\theta_{\mathbf{d}}(F(M)).$
\end{notation}
As a first instance of the convenience of notations just introduced, we can say a representation $M$ is $\mu$-semistable (resp. $\mu$-unstable) if and only if $G(M)=0$ (resp. $G(M)>0$) (see \Cref{connecting_slope_weight_stability}).

Now let $M=M^0$ be a representation of $Q$. If $G(M^0)=0$, then $M^0$ itself is the scss subrepresentation. Otherwise, let $M^1=F(M^0)$. \Cref{discrep_contains_1st_HN_term} then dictates that $G(M^1)=0$ if and only if $M^1$ is the scss subrepresentation of $M$. If $G(M^1)>0$, then $M^2=F(M^1)$ is a proper subrepresentation of $M^1$. We continue to produce $M^i=F(M^{i-1})$ if $G(M^{i-1})>0$. The procedure must end in a finite number of steps as $M$ is finite dimensional at each vertex, and each $M^{i}$ is a proper subrepresentation of $M^{i-1}$. At the end, we have a filtration $$0\subsetneq M^r\subsetneq M^{r-1}\subsetneq\cdots\subsetneq M^1\subsetneq M^0,$$ where 
\begin{enumerate}
    \item $\theta_{\mathbf{d}_{i-1}}(M^i)>0$,
    \item $M^{i}$ is $\theta_{\mathbf{d}_{i-1}}$-optimal in $M^{i-1}$, and 
    \item $M^{r}$ is $\mu$-semistable.
\end{enumerate}

The following proposition ensures that in this case $M^r$ is the scss subrepresentation of $M$. 

\begin{proposition}\label{Algorithm_produces_M1}
Let $0\subsetneq M^{r}\subsetneq\ldots\subsetneq M^{0}$ be a filtration of representations of $Q$ and let $\pmb{\dim}M^{i}=\mathbf{d}_{i}\in\mathbf{Z}^{Q_{0}}$. Suppose for each $i$, $$\theta_{\mathbf{d}_{i-1}}(M^{i})=\disc(M^{i-1},\theta_{\mathbf{d}_{i-1}})>0.$$  
Then $M^{r}$ is $\mu$-semistable if and only if $M^{r}$ is the scss subrepresentation of $M^{i}$ for $i=0,\ldots,r-1$.
\end{proposition}

\begin{proof}
The if part is trivial. The proof for the only if part can be carried out by induction. We let $M_1^{i}$ be the scss subrepresentation of $M^i$. Suppose $M^{r}$ is $\mu$-semistable. Then \Cref{discrep_contains_1st_HN_term} (applied to $\theta_{\mathbf{d}_{r-1}}$) implies $M^{r}=M^{r-1}_1$.  Now suppose $M^{r}=M^{i}_1$, we will show that $M^{r}=M^{i-1}_1$. For this, since $M^{i}$ is $\theta_{\mathbf{d}_{i-1}}$-optimal in $M^{i-1}$, \Cref{discrep_contains_1st_HN_term} implies $M^{i}$ contains $M^{i-1}_{1}$. This immediately implies $M^{i-1}_1=M^{i}_1$.  By the induction hypothesis, $M^r = M^{i}_1$ so that $M^{r}=M^{i-1}_{1}$, finishing the induction.
\end{proof}

We therefore propose the following algorithm to compute the Harder-Narasimhan filtration of a representation of an acyclic quiver. The outer while loop tests at each $i$-th step if $M/M_i$ is $\mu$-semistable. If not, the inner while loop then  computes the scss subrepresentation $M_{i+1}/M_i$ of $M/M_i$. The correctness of the inner loop is established by \Cref{Algorithm_produces_M1}. In the end of the algorithm we have $M_1$, $M_2/M_1,\ldots,M/M_{r-1}$ for some $r\geq 1$, where each $M_{i+1}/M_{i}$ is the scss subrepresentation of $M/M_{i+1}$, and $M/M_{r-1}$ is $\mu$-semistable. From these the Harder-Narasimhan filtration for $M$ is immediate.
\begin{algorithm}[H]
\begin{flushleft}
\algorithmicrequire{$\text{ A Representation }M\text{ of }Q$}
\algorithmicensure{$\text{ The Harder-Narasimhan filtration of }M$}\end{flushleft}
\begin{algorithmic}
\caption{An algorithm for Harder-Narasimhan filtrations}\label{HN_alg_psuedo_code}
\If {$G(M)=0$}
    \State {$\text{return }M$} \Comment{In this case $M$ is semistable, so the filration is $M$ itself.}
\EndIf
\While{$G(M)>0$}
    \State{$N\gets F(M)$}
    \While{$G(N)>0$} 
        \State $N\gets F(N)$
    \EndWhile
    \State $\text{Record }N$ \Comment{$N$ is the scss subrepresentation of $M$}
    \State $\text{Compute }M/N$
    \State $M\gets M/N$
\EndWhile
\end{algorithmic}
\end{algorithm}

We now show that the above algorithm satisfy the complexity bound given in the
\begin{theorem}\label{thmB_part_I}
Let $Q$ be an acyclic quiver and let $M$ be a representation of $Q$ over an infinite field. Let   $\Theta,\kappa:\mathbf{Z}^{Q_0}\rightarrow\mathbf{Z}$ be two weights that define the slope $\mu=\Theta/\kappa$, $P$ be the number of paths in $Q$, $\Omega=\sum_{v}|\Theta(v)|$, and let  $K=\kappa(M)$. \Cref{HN_alg_psuedo_code} constructs the Harder-Narasimhan filtration of $M$ within time complexity that is polynomial in $\Omega,K,$ and $P$. 
\end{theorem}

Note that this is only the first part of \Cref{Effective_alg_intro}. We postpone the precise statement and the proof of the second part to \Cref{one_PS_via_Kempf's_filtration} at the end of the next section.
\begin{proof}[proof of \Cref{thmB_part_I}]
For starters, in the $i$-th outer while loop of \Cref{HN_alg_psuedo_code} ($i$ starts from 0), the first line uses the algorithm of \Cref{discrepancy_rep_algorithm} to test $\mu$-stability of $M/M_{i}$. Let $\mathbf{d}^{i}$ be the dimension vector for the quotient $M/M_{i}$. \Cref{discrepancy_rep_algorithm} states that the time complexity to determine stability of $M/M_i$ is polynomial in $\sum_{v}|\theta_{\mathbf{d}^i}(v)|$, $\sum_{v}\dim (M/M_i)_{v}$, and $P$. Obviously the second term is bounded by $K$. To bound the first term by a polynomial in $\Omega$ and $K$, we first note that 
$$|\Theta(M/M_i)|\leq\sum_{v}|\Theta(v)|\dim(M/M_i)_v\leq \Omega K.$$ Therefore \begin{equation*}
\begin{split}
    \sum_v|\theta_{\mathbf{d}^{i}}(v)|&=\sum_v|\kappa(M/M_i)\Theta(v)-\Theta(M/M_i)\kappa(v)|\\
    &\leq \sum_v \kappa(M/M_i)|\Theta(v)|+|\Theta(M/M_i)|\kappa(v)\\
    &\leq \Omega K+\Omega K^2.
    \end{split}
\end{equation*} 
Hence, the first line of \Cref{HN_alg_psuedo_code} in the outer while loop has time complexity polynomial in $\Omega,K,P$. 

Next, in the $j$-th inner while loop, we are testing $\mu$-stability of some subrepresentation $M^{j}_i/M_i$ of $M/M_i$ using \Cref{discrepancy_rep_algorithm}. As before, the time complexity is bounded by a polynomial in $\Omega,K,P.$ Since there are at most $\sum_v\dim M_v\leq K$ many inner while loops, completing the inner while loop has again time complexity polynomial in $\Omega,K,P$. 

After the inner loop, we arrive at the scss subrpresentation $M_{i+1}/M_{i}$ of $M/M_{i}$. To compute a basis for the quotient $(M/M_{i+1})_v$ at each vertex $v$, we may compute a basis for a complement of $(M_{i+1}/M_{i})_v$ in $(M/M_{i})_v$. This can be done by computing the null space of the basis matrix for $(M_{i+1}/M_{i})_v$, which has complexity $O\big(\dim (M_{i+1}/M_i)_v^2\dim (M/M_i)_v\big).$ The complexity to compute bases at all vertices can therefore be bounded by $O(K^3)$. Next, for each arrow $a\in Q_1$, the map $(M/M_{i+1})(a)$ can be obtained via the following change of bases to the original map 
$$(M/M_{i})(a):(M_{i+1}/M_{i})_{ta}\oplus (M/M_{i+1})_{ta}\rightarrow (M_{i+1}/M_{i})_{ha}\oplus(M/M_{i+1})_{ha},$$ which has time complexity bounded by $O(K^3)$. Therefore, the time complexity to compute the quotient $M/M_{i+1}$ is bounded by a polynomial in $K$ and $P.$ 

In sum, we showed that each outer while loop has time complexity polynomial in $\Omega,K,P$. Since there can be at most $\sum_{v}\dim M_v\leq K$ outer loops, the total time complexity is still polynomial in $\Omega,K,P$.
\end{proof}

\section{Completing \texorpdfstring{\Cref{Effective_alg_intro}}{Theorem B}}\label{from_HN_to_OPS}
The major goal of this section is to complete \Cref{Effective_alg_intro}. That is, we are going to construct a maximally destabilizing one parameter subgroup from the Harder-Narasimhan filtration of an unstable representation. This is where GIT comes into play. 

GIT stability and its numerical criterion, known as the Hilbert-Mumford criterion (\Cref{HilbMumAffine}), were established by Mumford in \cite{MR1304906}. The Hilbert-Mumford criterion reduces testing stability for reductive group actions to one dimensional torus actions. This is done by restricting the action to one parameter subgroups of the reductive group. Mumford first  conjectured the existence of the one parameter subgroups that fail the numerical criterion maximally. The conjecture was then resolved by Kempf in his famous paper \cite{MR506989}. Beyond existence Kempf actually established the uniqueness of maximally destabilizing one parameter subgroups up to conjugacy by some parabolic subgroup (\Cref{theorem6}).

In \Cref{affine_GIT}, we will first recall Mumford's definition of stability (\Cref{Stability_def}) and the Hilbert-Mumford criterion (\Cref{HilbMumAffine}) in the affine setting. We then make precise the meaning of maximally destabilizing one parameter subgroups (\Cref{worst_one_PS_def}), and present Kempf's theorem (\Cref{theorem6}).

We then apply these machinery in the context of representations of quivers in \Cref{the_quiver_setting}. Finally in \Cref{the_complete_proof_of_thmB}, we recall in \Cref{one_PS_via_Kempf's_filtration} a way due to \cite{MR3199484} to construct a maximally destabilizing one parameter subgroup from the Harder-Narasimhan filtration of an unstable representation. This would complete \Cref{Effective_alg_intro}.

\subsection{Instability in affine geometric invariant theory}\label{affine_GIT}
Here we recall necessary notions and theorems from affine GIT. We work with a fixed algebraically closed field $k$. Let $G$ be a reductive group acting on an affine variety $X$. We let $\pmb{\Gamma}(G)$ denote the set of one parameter subroups of $G$. We fix a norm $||-||$ on $\pmb{\Gamma}(G)$. The norm would satisfy the following two properties:
\begin{enumerate}
    \item $||-||$ is invariant under conjugation. Namely, for any $g\in G$ and any $\lambda\in\pmb{\Gamma}(G)$, we have 
    $$||\lambda||=||g\lambda g^{-1}||.$$
    \item For any maximal torus $T\subset G$, the restriction of $||-||$ to the lattice $\pmb{\Gamma}(T)$ is induced by an inner product on the vector space $\pmb{\Gamma}(T)\otimes\mathbf{R}$ that is integral on $\pmb{\Gamma}(T)\times\pmb{\Gamma}(T)$.
\end{enumerate}
For the precise notion of a norm on the set of one parameter subgroups, we refer the readers to page 58 of \cite{MR1304906}. 

Stability in GIT depends on the choice of a linearized line bundle. In the affine setting, a character $\chi:G\rightarrow k^{\times}$ of $G$ induces a linearization of the trivial line bundle $\mathscr{O}_X$. Fix a character $\chi$.

For any one parameter subgroup $\lambda:k^{\times}\rightarrow G$ of $G$, we let $\langle\chi,\lambda\rangle$ be the integer that satisfies $(\chi\circ \lambda)(t)=t^{\langle\chi,\lambda\rangle}$ for all $t\in k^{\times}$. 
\begin{definition}\label{Stability_def}
An element $f\in k[X]$ is $\chi$-\emph{invariant of weight d} if $f(g\cdot x)=\chi^{d}(g^{-1})f(x)$ for all $g\in G$ and for all $x\in X$. We say a point $x\in X$
is $\chi$-\emph{semistable} if there is a $\chi$-invariant $f$ of positive
weight such that $f(x)\neq 0$. We say a point $x\in X$ is $\chi$-\emph{unstable} if $x$ is not $\chi$-semistable.
We write $X^{\text{ss}}(\chi)$ as the set of $\chi$-semistable points in $X$ and $X^{\text{us}}(\chi)$ as the complement $X-X^{\text{ss}}(\chi).$
\end{definition}

It follows from the definition that $X^{\text{ss}}(\chi)$ is a $G$-invariant open subvariety and that $X^{\text{us}}(\chi)$ is a $G$-invariant closed subvariety. Moreover, $X^{\text{ss}}(\chi)=X^{\text{ss}}(\chi^{d})$ for any $d>0$. 

\begin{remark}[\cite{MR1304906}] Let $k[X]_{\chi,d}$ be the space of $\chi$-invariant elements of weight $d$. The space $\oplus_{d\geq 0}k[X]_{\chi,d}$ has a natural graded ring structure. Let  $$X/\!\!/_{\chi}G:=\Proj(\oplus_{d\geq 0}k[X]_{\chi,d}).$$
Then there is a map $X^{\text{ss}}(\chi)\rightarrow X/\!\!/_{\chi}G$ that is constant on $G$-orbits, submersive and induces a bijection between  points in $X/\!\!/_{\chi}G$ and closed orbits in $X^{\text{ss}}(\chi)$. Moreover, $X/\!\!/_{\chi}G$ is a quasi-projective variety that is known as \emph{the GIT quotient of X by G with respect to }$\chi$. \end{remark}

We now introduce the Hilbert-Mumford criterion, which tests stability by restricting the group action to certain one parameter subgroups. Let $\lambda:k^{\times}\rightarrow G$ be a one parameter subgroup and let $x\in X$ be a point. We say $\lim\limits_{t\rightarrow 0}\lambda(t)\cdot x$ exists if the domain of the map $\lambda_{x}:k^{\times}\rightarrow X$ defined by $t\mapsto \lambda(t)\cdot x$ can be extended to the entire affine line. 
\begin{theorem}(Hilbert-Mumford criterion \cite{MR1304906}, \cite{MR1315461})\label{HilbMumAffine}
A point $x\in X$ is $\chi$-semistable if and only if for each one 
parameter subgroup $\lambda:k^{\times}\rightarrow G$ such that $\lim\limits_{t\rightarrow 0}\lambda(t)\cdot x$ exists,
we have $\langle\chi,\lambda\rangle\leq 0$.
\end{theorem}

Next, we define a numerical measure for instabilities contributed by destabilizing one parameter subgroups. 
For a point $x\in X$, set $$C_{x}=\{\lambda\in\pmb{\Gamma}(G)|\lim_{t\rightarrow
  0}\lambda(t)\cdot x\text { exits}\}.$$
According to \Cref{HilbMumAffine}, $x\in X^{\text{us}}(\chi)$ if and only if there is a one parameter subgroup $\lambda$ such that 
\begin{enumerate}
    \item $\lambda\in C_{x}$, and 
    \item $\langle\chi,\lambda\rangle>0$.
\end{enumerate}

Therefore, for an $x\in X^{\text{us}}(\chi)$, it is natural to ask if there is a one parameter subgroup that contributes to the highest instability measured by the quantities $\langle\chi,\lambda\rangle$ among all $\lambda\in C_{x}$. An immediate problem is that $\langle\chi,\lambda^{N}\rangle=N\cdot\langle\chi,\lambda\rangle$ for any $N\in\mathbf{N}$. To get rid of the dependency on multiples of one parameter subgroups, we divide the function $\langle\chi,-\rangle:\pmb{\Gamma}(G)\rightarrow\mathbf{Z}$ by the norm $||-||$ on  $\pmb{\Gamma}(G)$.

For $x\in X^{\text{us}}(\chi)$, we set $$\mathcal{M}^{\chi}(x):=
  \sup_{\lambda\in C_{x}\backslash\{0\}}\frac{\langle\chi,\lambda\rangle
}{||\lambda||}.$$ 
\begin{definition}\label{worst_one_PS_def} 
We say a one parameter subgroup $\lambda $ is \emph{indivisible} if $\lambda$ is a primitive lattice in a (and hence in any)  maximal torus  containing $\lambda$. 
We say a one parameter subgroup $\lambda$ is $\chi$-\emph{adapted to }$x$ if
$\frac{\langle\chi,\lambda\rangle}{||\lambda||}= \mathcal{M}^{\chi}(x)$. We let $\Lambda^{\chi}(x)$ denote the set of indivisible one parameter subgroups that are $\chi$-adapted to $x$.
\end{definition}

We also recall that for any one parameter subgroup $\lambda$, there is the associated parabolic subgroup $$P(\lambda)=\{g\in G\mid \lim_{t\rightarrow 0}\lambda(t)g\lambda(t)^{-1}\text{ exists in }G\}$$ of $G$. An important property about $P(\lambda)$ we want to mention is that it preserves filtration in the following sense: Let $V$ be a representation of $G$ and let $\oplus_{n\in\mathbf{Z}}V^{(n)}$ be the weight decomposition of some one parameter subgroup $\lambda$ of $G$. Namely, $\lambda$ acts on $V^{(n)}$ with weight $n$. Suppose $p\in P(\lambda)$, we then have $$p\cdot V^{(\geq n)}=V^{(\geq n)}$$ where $V^{(\geq n)}=\oplus_{m\geq n}V^{(m)}$. 

We are now ready for Kempf's theorem from \cite{MR506989}:
\begin{theorem}(Kempf)\label{theorem6}
Let $G$ be a reductive group, acting on an affine variety $X$. Let $\chi$ be a character of $G$, and let 
$x\in X$ be a $\chi$-unstable point. Then
\begin{enumerate}
\item $\Lambda^{\chi}(x)$ is not empty;
\item For any $g\in G$, $\Lambda^{\chi}(g\cdot x)=g\Lambda^{\chi}(x) g^{-1}$, and $\mathcal{M}^{\chi}(g\cdot x)=\mathcal{M}^{\chi}(x)$;
\item There is a
  parabolic subgroup $P(\chi,x)$ of $G$ such that for all $\lambda\in \Lambda^{\chi}(x)$, $P(\lambda)= P(\chi,x)$, and such that any two elements of $\Lambda^{\chi}(x)$ are conjugate to each other by an element in $P(\chi,x)$. 
\end{enumerate}
\end{theorem}
\subsection{The quiver setting}\label{the_quiver_setting}
We now unpack what these machineries mean in the case of representation of quivers. For starters, let us fix an acyclic quiver $Q$, and a dimension vector $\mathbf{d}\in\mathbf{Z}^{Q_{0}}$. By $\GL(n)$ we mean the linear algebraic group of invertible $n\times n$ matrices with entries in $k$. Let the linear algebraic group $$\GL(\mathbf{d}):=\prod_{v\in Q_0}\GL(\mathbf{d}_v)$$ act on the $k$-vector space $$R(Q,\mathbf{d}):=\prod_{a\in Q_1}\Hom_k(k^{\mathbf{d}_{ta}},k^{\mathbf{d}_{ha}})$$ via the following law: $$(g\cdot\phi)_a=g_{ha}\circ \phi_{a}\circ(g_{ta})^{-1}\text{ for any }g\in G, \phi\in R(Q,\mathbf{d}).$$ In this way, the $\GL(\mathbf{d})$-orbits in $R(Q,\mathbf{d})$ correspond to the isomorphism classes of representations of $Q$ of dimension $\mathbf{d}$. A weight $\theta:\mathbf{Z}^{Q_0}\rightarrow \mathbf{Z}$ defines a character $\chi_{\theta}:\GL(\mathbf{d})\rightarrow k^{\times}$ by $$\chi_{\theta}(g)=\prod_{v\in Q_0}\det(g_v)^{\theta(v)}\text{ for any }g\in \GL(\mathbf{d}).$$ We fix a weight $\theta$ with $\theta(\mathbf{d})=0$.

Let $M$ be a representation of $Q$ of dimension $\mathbf{d}$. Since $\chi_{\theta}$-stability is constant on $\GL(\mathbf{d})$-orbits, different choices of bases for each $M_v$ does not affect the $\chi_{\theta}$-stability of the corresponding elements of $M$ in $R(Q,\mathbf{d})$. From now on we fix an identification $M_v\simeq\mathbf{R}^{\mathbf{d}_v}$ for each $v$, and still write the identification of $M$ in $R(Q,\mathbf{d})$ as $M.$

We now recall the following relation between weighted filtrations of $M$, and the one parameter subgroups having limits at $M$, namely, the one parameter subgrous in $C_M$. Any one parameter subgroup $\lambda\in C_M$ induces a weighted filtration 
$$0=M_0\subsetneq M_1\subsetneq M_2\subsetneq\cdots \subsetneq M_s=M,$$ where $\lambda$ acts on each quotient $M_i/M_{i-1}$ with weight $\Gamma_i$. Moreover, the weights enjoy the following property:
 $\Gamma_1>\Gamma_2>\cdots>\Gamma_s$ (see \cite{MR1315461}). 

Conversely, given a weighted filtration $0=M_0\subsetneq M_1\subsetneq\cdots\subsetneq M_s=M$ with weights $\Gamma_1>\cdots>\Gamma_s$, we may choose new bases for $M$ at all vertices that are compatible with the filtration. Namely, choose a basis for $M_1$, then extends the basis to $M_2,M_3$, and so on. We then let $g_v\in\GL(\mathbf{d}_v)$ be the one parameter subgroup that acts on $(M_{i}/M_{i-1})_{v}$ by weight $\Gamma_i$. In this way, the one parameter subgroup $\lambda=\prod_{v\in Q_0}g_v$ of $\GL(\mathbf{d})$ induces the original filtration. Moreover, since $\theta(M)=0$, we have  \begin{equation}\label{hilbert-mum-index-original}\langle\chi_{\theta},\lambda\rangle=\sum_{i=1}^{s}\Gamma_i\theta(M_{i}/M_{i-1})=\sum_{i=1}^{s-1}(\Gamma_i-\Gamma_{i+1})\theta(M_i).\end{equation} 

In particular, for any proper subrepresentation $M'$ of $M$, one may form the filtration $0\subsetneq M' \subsetneq M$, and assign weights $\Gamma_1,\Gamma_2$ to $M'$ and $M$ respectively, where $\Gamma_1$ is any integer and $\Gamma_2=\Gamma_1-1$.  It then follows from the above discussion and \cref{hilbert-mum-index-original} that there is a one parameter subgroup $\lambda$ with $\langle\chi_{\theta},\lambda\rangle=\theta(M')$. In view of \Cref{HilbMumAffine}, we see that $M$ is $\chi_{\theta}$-semistable if and only if $\theta(M')\leq 0$ for any subrepresentation $M'$ of $M$. This was the definition, \Cref{weight_stability_def} we used for weight stability.

\subsection{From the Harder-Narasimhan filtration to maximally destabilizing one parameter subgroups}\label{the_complete_proof_of_thmB}
In this section, we address how to derive a maximally destabilizing one parameter subgroup of an unstable representation $M$ of $Q$ from its Harder-Narasimhan filtration. We need to set up appropriate weights and slope.

Let $\Theta,\kappa:\mathbf{Z}^{Q_0}\rightarrow\mathbf{Z}$ be two weights, and let $\mu=\Theta/\kappa$ be the associated slope function. Suppose $\pmb{\dim}M=\mathbf{d}.$ Recall that we defined the weight $\theta_{\mathbf{d}}$ where 
$$\theta_{\mathbf{d}}(N):=\kappa(\mathbf{d})\Theta(N)-\Theta(\mathbf{d})\kappa(N)\text{ for any representation }N.$$ With  these we have that $\theta_{\mathbf{d}}(M)=0$, and that $M$ is $\mu$-semistable if and only if it is $\chi_{\theta_{\mathbf{d}}}$-semistable (\Cref{connecting_slope_weight_stability}).

We now define a norm $||-||$ on $\pmb{\Gamma}(\GL(\mathbf{d}))$. This norm will depend on the weight $\kappa$. We first build the norm on the diagonal maximal torus $T\subset \GL(\mathbf{d})$. There is a natural identification $$\pmb{\Gamma}(T)_{\mathbf{R}}\simeq \bigoplus_{v\in Q_0}\mathbf{R}^{\mathbf{d}_v}.$$ We weight the standard norm on $\mathbf{R}^{\mathbf{d}_v}$ by $\kappa(v)$ for each $v$. More explicitly, if $\lambda\in \pmb{\Gamma}(T)$ with $\lambda_v=(\lambda_{1,v},\ldots,\lambda_{\mathbf{d}_v,v})$, we set  \begin{equation}\label{weighted_norm}||\lambda||=\big(\sum_{v\in Q_0}\sum_{i=1}^{\mathbf{d}_v}\kappa(v)(\lambda_{i,v})^{2}\big)^{1/2}.\end{equation} This norm $||-||$ extends to the entire set $\Gamma(\GL(\mathbf{d}))$ and is invariant under conjugation. We refer the readers to page 58 in \cite{MR1304906} for a detailed discussion on the construction of norms on the set of one parameter subgroup of a reductive group. 

Therefore, whenever $\lambda\in C_M$ (that is, $\lim_{t\rightarrow 0}\lambda(t)\cdot M$ exists), its norm can be expressed in terms of its weights and the filtration it induces: If $0=M_0\subsetneq M_1\subsetneq\cdots\subsetneq M_s=M$ is the filtration induced by $\lambda$ with weights $\Gamma_1>\cdots>\Gamma _s$, then \begin{equation}\label{weighted_norm_grouped_by_filtration}||\lambda||=\big(\sum_{i=1}^{s}\sum_{v\in Q_0}\kappa(v)\Gamma_i^2\dim (M_{i}/M_{i-1})_v\big)^{1/2}=\big(\sum_{i=1}^{s}\Gamma_i^2\kappa(M_{i}/M_{i-1})\big)^{1/2}.\end{equation}
We adopt the above norm to define (using \Cref{worst_one_PS_def}) one parameter subgroups that are $\chi_{\theta_{\mathbf{d}}}$-adapted to $M$.

Since all one parameter subgroups that are $\chi_{\theta_{\mathbf{d}}}$-adapted to $M$ are in a full conjugacy class of their parabolic subgroup (\Cref{theorem6}), the filtrations induced by those one parameter subgroups are the same. In \cite{MR3199484}, such a filtration is named as the Kempf filtration of $M$. By Theorem 5.3 from \cite{MR3199484}, it is also the Harder-Narasimhan filtration for $M$ with respect to the slope $\mu$. In addition, there is the following recipe for reverse engineering a maximally destabilizing one parameter subgroup from the Harder-Narasimhan filtration. 

\begin{theorem}[Theorem 4.1, Lemma 5.2 \cite{MR3199484}]\label{one_PS_via_Kempf's_filtration}
Let $\Theta,\kappa:\mathbf{Z}^{Q_0}\rightarrow\mathbf{Z}$ be two weights and let $\mu=\Theta/\kappa$ be the slope.  Let $M$ be a $\mu$-unstable representation of dimension $\mathbf{d}$, and let $$0=M_0\subsetneq M_1\subsetneq \cdots\subsetneq M_r=M$$ be the Harder-Narasimhan filtration. Set $$u_i=\kappa(M)\mu(M_{i}/M_{i-1})-\Theta(M)\text{ for }i=1,\ldots,r.$$ Choose a basis for $M_{v}$ for each vertex $v$ compatible with the filtration, and let $T\subset \GL(\mathbf{d})$ be the maximal torus diagonal with respect to these bases. Define $\tilde{\lambda}\in\pmb{\Gamma}(T)_{\mathbf{Q}}\simeq\oplus_{v\in Q_0}\mathbf{Q}^{\mathbf{d}_v}$ where for each $i$ and $v$, the entries of $\tilde{\lambda}_v$ that correspond to $(M_i/M_{i-1})_v$ are $u_i$. Then the lattice points on the ray $\mathbf{Q}_{>0}\cdot \tilde{\lambda}$ are the one parameter subgroups in $T$ that are $\chi_{\theta_{\mathbf{d}}}$-adapted to $M$. Moreover, $$\mathcal{M}^{\chi_{\theta_{\mathbf{d}}}}(M)=\big(\sum_{i=1}^r(\kappa(M)\mu(M_{i}/M_{i-1})-\Theta(M))^2\kappa(M_i/M_{i-1})\big)^{1/2} .$$ 
\end{theorem}
\begin{remark}\label{one_PS_via_Kempf's_filtration_simpli_version}
In the case that $\Theta(M)=0$ in the first place, one can take $$u_i=\mu(M_i/M_{i-1}),$$ and \begin{equation*}\begin{split}
\mathcal{M}^{\chi_{\Theta}}(M)&=\big(\sum_{i}\mu(M_{i}/M_{i-1})^2\kappa(M_{i}/M_{i-1})\big)^{1/2}\\
&=\big(\sum_{i}\Theta(M_{i}/M_{i-1})\mu(M_{i}/M_{i-1})\big)^{1/2}.\end{split}\end{equation*}\end{remark}

Therefore, knowing the Harder-Narasimhan filtration of an unstable representation is equivalent to knowing its  maximally destabilizing one parameter subgroups. A full treatment of \Cref{Effective_alg_intro} is now complete. An account of the correctness of the construction in \Cref{one_PS_via_Kempf's_filtration} is given in the next section for interested readers.
\subsection{Maximizing Kempf function}\label{maximizing_Kempf_function}
Let $M$ be a $\mu$-unstable representation of $Q$ of dimension $\mathbf{d}$ with the Harder-Narasimhan filtration 
$$0=M_0\subsetneq M_1\subsetneq\cdots\subsetneq M_r=M.$$ By Theorem 5.3 from \cite{MR3199484}, the filtration is the same as the Kempf filtration. Recall that $C_M$ denotes the set of one parameter subgroups having limits at $M$. Let $$f:C_M\backslash\{\mathbf{0}\}\rightarrow\mathbf{R}$$ be the function defined by $$f(\lambda)=\frac{\langle\chi_{\theta_{\mathbf{d}},\lambda}\rangle}{||\lambda||}.$$ To maximize this function, we can restrict the search range to the set of one parameter subgroups that induce the Kempf filtration, which is also the Harder-Narasimhan filtration. 

Consider the inner product $(-,-):\mathbf{R}^{r}\times\mathbf{R}^{r}\rightarrow\mathbf{R}$ defined by the matrix 
$$
\begin{pmatrix}
\kappa(M_1)&&&0\\
&\kappa(M_2/M_1)&&\\
&&\ddots&\\
0&&&\kappa(M/M_{r-1})
\end{pmatrix}.
$$
If $\lambda$ is a one parameter subgroup that induces the Harder-Narasimhan filtration of $M$ with the weight vector $\Gamma=(\Gamma_1,\ldots,\Gamma_s)$, then 
upon using $\theta_{\mathbf{d}}$ as $\theta$ in equation (\ref{hilbert-mum-index-original}),  we derive that 
\begin{equation}\label{hilbert-mumford-index-weighted}
\begin{split}
\langle\chi_{\theta_{\mathbf{d}}},\lambda\rangle&=\sum_{i}\Gamma_{i}\big(\kappa(M)\Theta(M_i/M_{i-1})-\Theta(M)\kappa(M_i/M_{i-1})\big)\\
&=\sum_i\Gamma_i\kappa(M_i/M_{i-1})\big(\kappa(M)\mu(M_i/M_{i-1})-\Theta(M)\big)\\
&=(\Gamma,u),
\end{split}
\end{equation}
where $u_i=\kappa(M)\mu(M_i/M_{i-1})-\Theta(M)$. Moreover, \cref{weighted_norm_grouped_by_filtration}  implies $||\lambda||=(\Gamma,\Gamma)^{1/2}$. 

Letting $u=(u_1,\ldots,u_r)\in\mathbf{R}^{r}$, we define the function $$g_{u}:\mathbf{R}^{r}\backslash\{\mathbf{0\}}\rightarrow\mathbf{R}$$ where $$g_u(x)=\frac{(x,u)}{\sqrt{(x,x)}}.$$ This function is called the Kempf function in \cite{MR3199484}. We see that maximizing  $f$ among all one parameter subgroups inducing the Kempf filtration is equivalent to maximizing the Kempf function $g_u$ on the open cone $C=\{x\in\mathbf{R}^{r}\mid x_1>x_2>\cdots>x_r\}.$ By the Cauchy-Schwarz inequality, $g_u$ attains global maximum at $u$. Since $\mu(M_1)>\mu(M_2/M_1)>\cdots>\mu(M/M_r)$, $u$ is already in the cone $C$. Therefore, the one parameter subgroup constructed in \Cref{one_PS_via_Kempf's_filtration} maximizes $f$.

\section{A short example}\label{a_short_example}
In this example we construct an unstable representation of a bipartite quiver and compute the following:
\begin{enumerate}
\item Several witnesses to the discrepancy;
\item The Harder-Narasimhan filtration;
\item A maximally destabilizing one parameter subgroup.
\end{enumerate}
We compare the results of (1), (2) with \Cref{HN_contains_witness_discrepancy}, and demonstrate that not all optimal subrepresentations occur in the Harder-Narasimhan filtration. For (3), we present two approaches. The first approach applies \Cref{one_PS_via_Kempf's_filtration} to the Harder-Narasimhan filtration obtained in (2), and the second is a sketch of the brute force procedure that can be applied to any representation of any reductive group. By doing so we hope to give readers a sense of the hardeness of finding Kempf's one parameter subgroups in general. The set up of this example is motivated by \cite{MR4238986}. 
\subsubsection{Set up}
Let us consider the following representation $M$ of the following bipartite quiver over the field $\mathbf{C}$ of complex numbers:
$$
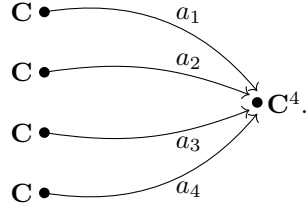
\begin{tikzpicture}
\node at (-0.3,0.8) {$\mathbf{C}$};
\node at (-0.3,0) {$\mathbf{C}$};
\node at (-0.3,-0.8) {$\mathbf{C}$};
\node at (-0.3,-1.6) {$\mathbf{C}$};
\node at (3.2,-0.4) {$\mathbf{C}^4$.};
\fill (0,0.8) circle (2pt);
\fill (2.8,-0.4) circle (2pt);
\fill (0,0) circle (2pt);
\fill (0,-0.8) circle (2pt);
\fill (0,-1.6) circle (2pt);
\draw [->] (0,0.8) to [bend left] (2.8,-0.25);
\draw [->] (0,0) to [bend left=15] (2.7,-0.3);
\draw [->] (0,-0.8) to [bend right = 15] (2.7,-0.5);
\draw [->] (0,-1.6) to [bend right] (2.8,-0.55);
\node at (1.9,0.75) {$a_1$};
\node at (1.9,0.15) {$a_2$};
\node at (1.9,-0.95) {$a_3$};
\node at (1.9,-1.55) {$a_4$};
\end{tikzpicture}$$
Label the four vertices on the left by $x_1,x_2,x_3,x_4$ from top to bottom. Label the single vertex on the right as $y$ so that $M_{x_i}=\mathbf{C}$ for each $i=1,\ldots ,4$, and $M_{y}=\mathbf{C}^4$. 

We set $M(a_1)(1)=e_{2}, M(a_2)(1)=e_{1}, M(a_{3})(1)=2e_{1}$, and  $M(a_{4})(1)=e_{3}$ in $\mathbf{C}^{4}$, where $e_1,\ldots,e_4$ form the standard basis of $\mathbf{C}^4$. We let $\Theta$ and $\kappa$ be two weights where $\Theta(x_{i}) = 4$, $\Theta(y)=-4$, and $\kappa(x_i)=\kappa(y) = 1$ for each $i$. It then follows that $\Theta(M) = 0$. We also let $\mu(-) = \Theta(-)/\kappa(-)$ be the corresponding slope. 

We let $\Sp(e_1)$ be the largest subrepresentation $M'$ of $M$ such that $M'_y=\Span(e_1)$. Specifically,  $$\Sp(e_1)_{x_{i}}=\begin{cases}
\mathbf{C}\text{ if }i=2,3,\\
0\text{ if i=1,4},
\end{cases}$$
and $$\Sp(e_1)_y=\Span(e_1).$$ We define $\Sp(e_1,e_2),\Sp(e_1,e_3),$ and $\Sp(e_1,e_2,e_3)$ similarly. 
\subsubsection{Witnesses to the discrepancy}
We verify that the four subrepresentations $\Sp(e_{1})$, $\Sp(e_{1},e_{2})$, $\Sp(e_{1},e_{3})$, $\Sp(e_{1},e_{2},e_{3})$ are $\Theta$-optimal in $M$ with $\Theta$ value 4.   
For this, note that for any subrepresentation $M^{\prime}$ of $M$, $$\Theta(M^{\prime}) \leq 4\cdot \big(\# \{i\mid M'_{x_i}=\mathbf{C}\}-\dim(\Span(\{M(a_i)(1) \mid M'_{x_i}=\mathbf{C}\}))\big).$$ Therefore, for $\Theta(M^{\prime})$ to be positive, $M'_{x_2}$ and $M'_{x_3}$ must be equal to $\mathbf{C}$. The only proper subrepresentations of $M$ that are nonzero at $x_2,x_3$ with positive $\Theta$ values are exactly the four subrepresentations listed above. Since they all have the same $\Theta$ value 4, they are also $\Theta$-optimal and $\disc(M,\Theta)=4$. 
\subsubsection{The Harder-Narasimhan filtration}
We verify that 
$$0\subsetneq \Sp(e_1)\subsetneq \Sp(e_1,e_2,e_3)\subsetneq M$$ is the Harder-Narasimhan filtration of $M$. 
For this, it is clear that $\Sp(e_1)$ is the only subrepresentation with the highest slope $4/3$. We may infer that $\Sp(e_1)$ is the scss subrepresentation of $M$. Moving on to $M/\Sp(e_1)$, since $\Sp(e_1)$ is $\Theta$-optimal,  we can infer $\mu(M'/\Sp(e_{1}))\leq 0$ for any subrepresentation $M'$ containing $\Sp(e_1)$. Since  $$\mu(\Sp(e_1,e_2)/\Sp(e_1))=\mu(\Sp(e_1,e_3)/\Sp(e_1))=\mu(\Sp(e_1,e_2,e_3)/\Sp(e_1))=0,$$ we infer that the quotient $\Sp(e_1,e_2,e_3)/\Sp(e_1)$ is the scss subrepresentation of $M/\Sp(e_1)$. The quotient $M/\Sp(e_1,e_2,e_3)$ is one dimensional on $y$ and zero elsewhere. We conclude that $M/\Sp(e_1,e_2,e_3)$ is $\mu$-semistable and that the filtration written above is the Harder-Narasimhan filtration of $M$. 

Here we note that the two $\Theta$-optimal subrepresentations $\Sp(e_1,e_2)$ and $\Sp(e_1,e_3)$ are not included in the filtration.  We see in this example that not all optimal subrepresentations occur in the Harder-Narasimhan filtration. Also note that $\Sp(e_1)$ has the highest slope among all $\Theta$-optimal subrepresentations. This agrees with the implications of \Cref{HN_contains_witness_discrepancy} because $\mu(\Sp(e_1))=4/3>0=\mu(\Sp(e_1,e_2,e_3)/\Sp(e_1))$. Namely, $\Sp(e_1)$ serves as $M_l$ in the description of  \Cref{HN_contains_witness_discrepancy}.

\subsubsection{Maximal destabilizing one parameter subgroups}
We now demonstrate a maximally destabilizing one parameter subgroup for $M$ with respect to the weight $\Theta$ and the norm weighted by $\kappa$. The first thing to do is to reproduce the GIT set up introduced in \Cref{the_quiver_setting}. To begin with, we have the vector space $\prod_{a\in Q_1}\Hom_\mathbf{C}(\mathbf{C},\mathbf{C}^4)\simeq \mathbf{C}^{4\times 4}$. Each column of a matrix in $\mathbf{C}^{4\times 4}$ defines a map from a left vertex to the right vertex. We let the first column define a map from $x_1$ to $y$, and so on. In this way, the group $G=(\mathbf{C}^{\times})^{4}\times \GL(4)$ acts on $\mathbf{C}^{4\times 4}$ via the rule
$$(t_1,t_2,t_3,t_4,A)\cdot U=A\cdot U\cdot \begin{pmatrix}
t_1^{-1}&&&\\
&t_2^{-1}&&\\
&&t_3^{-1}&\\
&&&t_4^{-1}
\end{pmatrix}$$ for all $(t_1,t_2,t_3,t_4,A)\in G$ and for all $U\in\mathbf{C}^{4\times 4}$. There is the diagonal torus $D\simeq (\mathbf{C}^{\times})^{4}$ in $\GL(4)$. Let $T=(\mathbf{C}^{\times})^{4}\times D\subset G$ be the maximal torus in $G$. We then have $\pmb{\Gamma}(T)_{\mathbf{Q}}\simeq\mathbf{Q}^{8}$. With the notations introduced in \Cref{one_PS_via_Kempf's_filtration}, we have $$u_1=\frac{4}{3},u_2= 0, u_3=-4$$ by \Cref{one_PS_via_Kempf's_filtration_simpli_version}. The maximally destabilizing one parameter subgroups in $T$ is then on the positive  ray of the point $$(\underbrace{0,\frac{4}{3},\frac{4}{3},0}_{x_1,x_2,x_3,x_4},\underbrace{\frac{4}{3},0,0,-4}_{y}).$$

We now sketch a procedure using linear programming. Here we do not assume we know the Harder-Narasimhan filtration in the first place. Recall that $\Theta$ defines the character $\chi_{\Theta}$ of $G$ where $$\chi_{\Theta}(t_1,t_2,t_3,t_4,A)=(t_1t_2t_3t_4)^4\cdot \det(A)^{-4} $$ for all $(t_1,t_2,t_3,t_4,A)\in G$. The representation $M$ corresponds to the matrix $\begin{pmatrix}
0&1&2&0\\
1&0&0&0\\
0&0&0&1\\
0&0&0&0
\end{pmatrix}.$ Let us continue to use $M$ to represent the matrix.

If $\lambda(t)=(t^{a_1},\ldots, t^{a_8})$ is a one parameter subgroup in $T$, then $$\lambda(t)\cdot M=
\begin{pmatrix}
0&t^{-a_2+a_5}&2\cdot t^{-a_3+a_5}&0\\
t^{-a_1+a_6}&0&0&0\\
0&0&0&t^{-a_4+a_7}\\
0&0&0&0
\end{pmatrix}.$$ We see that $\lim_{t\rightarrow\infty}\lambda(t)\cdot M$ exists if and only if 
\begin{equation}\label{constraints_lp_short_example}\begin{cases}
-a_1+a_6\geq 0\\
-a_2+a_5\geq 0\\
-a_3+a_5 \geq 0\\
-a_4+a_7\geq 0\end{cases}.\end{equation}
One also easily calculates that $$\langle\chi_{\Theta},\lambda\rangle= 4(a_1+a_2+a_3+a_4)-4(a_5+a_6+a_7+a_8).$$

The norm on the set of one parameter subgroups of $T$ is weighted by $\kappa$. In this case it is simply the standard norm. Therefore, to find one parameter subgroups in $T$ that maximally destabilize $M$, we are maximizing the function 
$$f=\frac{4(a_1+a_2+a_3+a_4)-4(a_5+a_6+a_7+a_8)}{(a_1^2+a_2^2+a_3^2+ a_4^2+a_5^2+a_6^2+a_7^2 +a_8^2)^{1/2}},$$ with the side constraints given by \cref{constraints_lp_short_example}. With the help of SageMath, we calculated that $f$ attains maximum on the positive ray of the point $(0,\frac{4}{3},\frac{4}{3},0,\frac{4}{3},0,0,-4)$. 

Notice that the constraints depends only on the nonzero entries of $M$. That is, replacing the nonzero entries of $M$ by any other nonzero scalars does not change the linear programming problem. The set of the positions of the nonzero entries of $M$ is called the support of $M$.

We must be aware that up to this point, we are only calculating one parameter subgroups in $T$. In general, maximally destabilizing one parameter subgroups of an unstable point do not have to be in a particular maximal torus. To work with a particular maximal torus, one needs to enumerate all supports coming from the orbit of the unstable point,  conduct linear programming for each support, and choose the maximum result. Often times the challenge is to determine the supports of points in the orbit. Beyond that the number of supports of points in the orbit can be quite large, leading to vast amount of linear programming required. As we can see, enumerating supports of $g\cdot M$ for all $g\in G$ in this relatively small example already gets tedious. We shall not continue this brute force procedure here, as the main point to illustrate the hardness of finding Kempf's one parameter subgroups in general, and to demonstrate the value of \Cref{one_PS_via_Kempf's_filtration} is made. 
\bibliography{bibitem}
\bibliographystyle{alpha}
\appendix

\section{Algebraic complexity towards the stability of representations of quivers}\label{appendixB}
In this section, we will demonstrate how algebraic complexity tools can be applied to computing the discrepancy, together with a witnesses, for a representation of an acyclic quiver.
\subsection{Preliminaries}
For starters, we let $\mathbf{F}$ be an infinite field of arbitrary characteristic. The space of $n\times n$ matrices over $\mathbf{F}$ is denoted $M(n,\mathbf{F})$.
Let $\mathcal{B}$ be a subspace of $M(n,\mathbf{F})$. If $U$ is a subspace of $\mathbf{F}^n$, we set $$\mathcal{B}(U)=\sum_{B\in\mathcal{B}}BU,$$ where $\sum$ denote the sum of subspaces.  We say $U$ is a $c$-\emph{shrunk subspace of }$\mathcal{B}$ if $$\dim U-\dim\mathcal{B}(U)\geq c.$$ We say $U$ is a \emph{shrunk subspace of }$\mathcal{B}$ if $U$ is a $c$-shrunk subspace of $\mathcal{B}$ for some $c\in\mathbf{Z}^{+}$. We define the \emph{discrepancy of }$\mathcal{B}$ as the integer  
$$\disc(\mathcal{B})=\max_{c\in\mathbf{N}}\{\exists\text{ a }c\text{-shrunk subspace of }\mathcal{B}\}.$$

\begin{remark}\label{existence_of_minimality_c_shrunk}
For any subspace $\mathcal{B}\subset M(n,\mathbf{F})$, if $c=\disc(\mathcal{B})$, it is shown in \cite{https://doi.org/10.48550/arxiv.2111.00039}, Lemma 2.1 that the intersection of $c$-shrunk subspaces is again a $c$-shrunk subspace. Therefore, there is an unique minimal $c$-shrunk subspace of $\mathcal{B}$.
\end{remark}

The following is a summary of various results from \cite{MR3868734} and \cite{MR3354797} .
\begin{theorem}[Theorem 1.5 \cite{MR3868734}, Proposition 7, Lemma 9 \cite{MR3354797}]\label{thm_1.5_ISQ18}
Let $\mathcal{B}$ be a matrix space in $M(n,\mathbf{F})$ with $\disc(\mathcal{B})=c$ (a priori unknown). There exists a deterministic algorithm using $n^{O(1)}$ arithmetic operations over $\mathbf{F}$ that returns a $c$-shrunk subspace of $\mathcal{B}$. Moreover, the $c$-shrunk subspace produced by this algorithm is the minimal $c$-shrunk subspaces of $\mathcal{B}$.
\end{theorem}

\begin{remark}
The original Theorem 1.5 from \cite{MR3868734} deals with how field elements are represented in a computer, which is not the main concern of this paper. Personal communications with the authors of \cite{MR3868734} confirmed that in terms of the number of operations needed, there is nothing special about the field of rational numbers stated in the theorem in \cite{MR3868734}. Any field with large enough set of elements to work with will suffice.
\end{remark}

We now follow the recipe of \cite{https://doi.org/10.48550/arxiv.2111.00039} to build a matrix space from a given representation of an acyclic quiver, and establish the connection between the discrepancy of the matrix space and that of the representation. Recall we wrote a quiver as $Q=(Q_0,Q_1,t,h)$. Multiple arrows between any two vertices are allowed. 
We fix a weight $\theta:Q_0\rightarrow\mathbf{Z}$ and a representation $W$ of $Q$ of dimension $\mathbf{d}$ with $\theta(\mathbf{d})=0$. If $p=a_ja_{j-1}\cdots a_1$ is a path in $Q$, we let $W(p)$ denote the composition $W(a_j)W(a_{j-1})\cdots W(a_1)$. 

Let $x_1,\ldots,x_n$ (resp. $y_1,\ldots,y_m$) be the vertices of $Q$ where $\theta$ takes positive values (resp. negative values). We then define $$\theta_{+}(x_i)=\theta(x_i)\text{ }(\text{resp. }\theta_{-}(y_j)=-\theta(y_j)),$$ 
$$N:=\sum_{i}\theta_{+}(x_i)\mathbf{d}(x_i)=\sum_{j}\theta_{-}(y_j)\mathbf{d}(y_j),$$ and $$M=\sum_{j}\theta_{-}(y_j)\text{ }(\text{resp. }M'=\sum_{i}\theta_{+}(x_i)).$$ For each $i\in[n]$ and each $j\in [m]$, we let 
$$I^{+}_{i}=\{r\in\mathbf{N}\mid \sum_{k=1}^{i-1}\theta_{+}(x_i)< r\leq \sum_{k=1}^{i}\theta_{+}(x_i)\},$$ and 
$$I_{j}^{-}=\{q\in\mathbf{N}\mid \sum_{k=1}^{j-1}\theta_{-}(y_j)< q  \leq\sum_{k=1}^{j}\theta_{-}(y_j)\}.$$

For each $q\in I_{j}^{-}$, $r\in I_{i}^{+}$, and each path $p$ between $x_i$ and $x_j$, we define the $M\times M'$ block matrix $A^{i,j,p}_{q,r}$ whose $(q,r)$-block is the $\mathbf{d}(y_j)\times \mathbf{d}(x_i)$ matrix representing $W(p):\mathbf{R}^{d(x_i)}\rightarrow\mathbf{R}^{d(y_j)}$, and whose other blocks are zero block matrices of appropriate sizes. In this way, the same matrix representing $W(p)$ is placed in the various $(q',r')$-blocks of $M\times M'$ block matrices for $q'\in I_{j}^{-}$, and  $r'\in I_{i}^{+}$. We set $$\mathcal{A}_{W,\theta}=\Span(\{A^{i,j,p}_{q,r}\mid i\in[n], j\in[m], q\in I_{j}^{-},r\in I_{i}^{+}, p\in \mathcal{P}_{i,j}\})\subset M(N,\mathbf{F}).$$
where $\mathcal{P}_{i,j}$ is the collection of paths from $x_{i}$ to $y_j$. 

We have an immediate relation described by the 
\begin{lemma}\label{rep_disc_bounded_by_mat_disc}
For any representation $W$, we have $\disc(W,\theta)\leq \disc(\mathcal{A}_{W,\theta})$. 
\end{lemma}

\begin{proof}
Let $W'\subset W$ be a subrepresentation. Define $U=\bigoplus_{i}(W'_{x_i})^{\theta_{+}(x_i)}$. It then follows that 
\begin{equation*}\begin{split}\theta(W')&=\sum_{i}\theta_{+}(x_i)W'_{x_i}-\sum_j\theta_{-}(y_j)W'_{y_j}=\dim U-\dim\big(\bigoplus_j W_{y_j}^{\theta_-(y_j)}\big)\\
&\leq \dim U-\dim\big(\bigoplus_{j,q}\sum_{i}\sum_{p\in \mathcal{P}_{ij}}W(p)(W'_{x_i})\big)=\dim U-\dim \mathcal{A}_{W,\theta}(U).\end{split}\end{equation*}
This finishes the proof of the lemma.
\end{proof}

\subsection{Relating the two discrepancies}
In this section we relate better the two discrepancies $\disc(W,\theta)$, and $\disc(\mathcal{A}_{W,\theta})$. We will show that they are equal, and demonstrate how to translate a witness to one discrepancy into a witness to the other. For this we need to introduce a right action on $\bigoplus_iW^{\theta_{+}(x_i)}_{x_i}$ and on $\bigoplus_{j}W_{y_j}^{\theta_{-}(y_j)}$ by the paths of $Q$. For the following, homomorphisms and endomorphisms are taken in the category of representations of $Q$. In addition, we let $c=\disc(\mathcal{A}_{W,\theta})$ so that by a $c$-shrunk subspace of $\mathcal{A}_{W,\theta}$ we mean a subspace that witnesses $\disc(\mathcal{A}_{W\,\theta})$.

For any vertex $x\in Q_0$, we let $P_x$ be the representation of $Q$ where for each $y\in Q_0$, $(P_x)_y$ is the vector space with basis the paths from $x$ to $y$. Here $(P_x)_x$ is one dimensional with basis the trivial path $e_x$. For any $a\in Q_1$, we define $$P_x(a):(P_x)_{ta}\rightarrow (P_x)_{ha}$$ by post composition. 

For any representation $V$ of $Q$, there is a natural identification $\Hom(P_x,V)\simeq V_x$. Letting $$P_1=\bigoplus_i P_{x_i}^{\theta_{+}(x_i)}\text{ } \big(\text{resp. }P_0=\bigoplus_j P_{y_j}^{\theta_{-}(y_j)}\big),$$ we have 
\begin{equation*}
\Hom(P_1,W)\simeq \bigoplus_i W_{x_i}^{\theta_{+}(x_i)}\text{ }\big(\text{resp. }\Hom(P_0,W)\simeq \bigoplus_{j}W_{y_j}^{\theta_{-}(y_j)}\big).\end{equation*}
Under the above identifications, a morphism $\varphi\in\Hom(P_0,P_1)$ induces a morphism $A(\varphi):\bigoplus_{i}W_{x_i}^{\theta_{+}(x_i)}\rightarrow \bigoplus_{j}W_{y_j}^{\theta_{-}(y_j)}$ by pre-composition. Moreover, if $p$ is a path from $x_i$ to $y_j$ placed in the $(q,r)$-th component of $\Hom(P_0,P_1)\simeq\bigoplus_{i,j}P_{x_i}(y_j)^{\theta_{-}(y_j)\theta_{+}(x_i)},$ then the map it induces is precisely $A^{i,j,p}_{q,r}$. Therefore, for any subspace $U\subset \bigoplus_{i}W_{x_i}^{\theta_{+}(x_i)}$, $\mathcal{A}_{W,\theta}(U)$ is the same as the sum $\sum_{\varphi}A(\varphi)(U)$, taken over all $\varphi\in\Hom(P_0,P_1)$. 

Similarly, $\Endo(P_1)$ (resp. $\Endo(P_0)$) induces a right action on $\bigoplus_{i}W_{x_i}^{\theta_{+}(x_i)}$ (resp. $\bigoplus_{j}W_{y_j}^{\theta_{-}(y_j)}$) by pre-composition. We recall \Cref{existence_of_minimality_c_shrunk} that there is a unique minimal subspace that witnesses $\disc(\mathcal{A}_{W,\theta}).$ Now we quote the 
\begin{lemma}[Lemma 3.2 \cite{https://doi.org/10.48550/arxiv.2111.00039}]
Let $U\subset\bigoplus_{i}W_{x_i}^{\theta_{+}(x_i)}$ be the minimal $c$-shrunk subspace of $\mathcal{A}_{W\,\theta}$. Then $U$ (resp. $\sum_{\varphi}A(\varphi)(U)$) is a right $\Endo(P_1)$ (resp.  $\Endo(P_0)$)-module. Namely, $U$ (resp. $\sum_{\varphi}A(\varphi)(U)$) is closed under the right action of $\Endo(P_1)$ (resp. $\Endo(P_0)$) defined by pre-composition.
\end{lemma}
We now apply the lemma to establish the 
\begin{proposition}\label{subrep_preconstruction}
Let $U$ be the minimal $c$-shrunk subspace of $\mathcal{A}_{W,\theta}$, and let $V=\sum_{\varphi}A(\varphi)(U)$. For each $i,r$ (resp. $j,q$), let $\pi_{i,r}:\bigoplus_{i}W_{x_i}^{\theta_{+}(x_i)}\rightarrow W_{x_i}$ (resp. $\pi_{j,q}: \bigoplus_{j}W_{y_j}^{\theta_{-}(y_j)}\rightarrow W_{y_j}$) be the projection onto the $(i,r)$-th (resp. $(j,q)$-th) component.
The following properties then hold:
\begin{enumerate}
    \item $U=\bigoplus_{i}(W'_{x_i})^{\theta_{+}(x_i)}$, where $W'_{x_i}=\pi_{i,r}(U)\subseteq W_{x_i}$ for any $r\in I^{+}_{i}$.
    \item $V=\bigoplus_{j}(W'_{y_j})^{\theta_{-}(y_j)}$, where $W'_{y_j}=\pi_{j,q}(V)=\sum_{i,p\in\mathcal{P}_{i,j}}W(p)(W'_{x_i})\subseteq W_{y_j}$ for any $q\in I_{j}^{-}$.
\end{enumerate}
\end{proposition}
\begin{proof}
Let $u\in U$ and let $p$ be a path from $x_{i'}$ to $x_{i}$ in the $(r',r)$-th component of $\Endo(P_1)\simeq \bigoplus_{i,i'}(P_{x_{i'}})_{x_i}^{\theta_{+}(x_{i'})\theta_{+}(x_{i})}.$ The pre-composition $u\circ p$ is then $W(p)(\pi_{i',r'}(u))$ placed in the $r$-th component of $W(x_{i})^{\theta_{+}(x_i)}$. Since $U$ is a right $\text{End}(P_{1})$-module, and the same path $p$ can be placed in any $(r',r)$-th component of $\Endo(P_1)$, the image $W(p)(\pi_{i'r'}(u))$ when placed in the $r$-th component of $W(x_{i})^{\sigma_{+}(x_i)}$, is still in $U$ for any $r'\in I^{+}
_{i'}$ and any $r\in I^{+}_{i}$. In particular, taking $i'=i$, we see that for any $r_1,r_2\in I_{i}^{+}$, $\pi_{i,r_1}(u)$ when placed at the $r_2$-th component of $W_{x_i}^{\theta_{+}(x_i)}$, is still in $U$. Basically, we obtain $$\pi_{i,r_1}(U)=\pi_{i,r_2}(U)\text{ for any }r_1,r_2\in I_{i}^{+}\text{, and }$$ $$U=\bigoplus_{i}\big(\pi_{i,r}(U)\big)^{\theta_{+}(x_i)}\text{ for any r}\in I_{i}^{+}.$$ This establishes property (1).  

Using the fact that $V$ is a right $\Endo(P_0)$-module, we also get $$V=\bigoplus_{j}\big(\pi_{j,q}(V)\big)^{\theta_{-}(y_j)}\text{ for any q}\in I^{-}_{j}.$$  
On the other hand, it is clear that \begin{equation*}
\begin{split}
\sum_{\varphi}A(\varphi)(U)&=\bigoplus_{j,q}\big(\sum_{i,p}\sum_{r}A^{i,j,p}_{q,r}\cdot \pi_{i,r}(U)\big)=
\bigoplus_{j,q}\big(\sum_{i,p}\sum_{r}W(p)(\pi_{i,r}(U)\big)\\
&=\bigoplus_{j}\big(\sum_{i,p\in\mathcal{P}_{i,j}}W(p)(W'_{x_i})\big)^{\theta_{-}(y_j)}
\end{split}\end{equation*}
This establishes property (2), finishing the proof of the proposition.
\end{proof}

\begin{theorem}\label{from_mat_disc_to_rep_disc}
Let $U\subset M(N,\mathbf{F})$ be the minimal $c$-shrunk subspace of $\mathcal{A}_{W,\theta}$. Let $\pi_{i,r}:\bigoplus_{i}W_{x_i}^{\sigma_{+}(x_i)}\rightarrow W_{x_i}$ be the projection onto the $(i,r)$-th component. 
For each $i\in[n]$, set $$W'_{x_i}=\pi_{i,r}(U).$$ For each $j\in[m]$, set $$W'_{y_j}=\sum_{i,p\in\mathcal{P}_{i,j}}W(p)\big(W_{x_i}\big).$$ For each $y\in Q_0$ with $\theta(y)=0$, set 
$$W'_{y}=\sum_{\theta(x)\neq 0}\bigg(\sum_{p:x\rightarrow y}W(p)\big(W_{x}\big)\bigg).$$ Then $\{W'_{x}\}_{x\in Q_0}$ defined this way yields a subrepresentation $W'$ with $\theta(W')=c$. In this case,  we actually have $$\disc(W,\theta)=c.$$
\end{theorem}

\begin{proof}
We first verify that $\{W'_{x}\}_{x\in Q_0}$ forms a subrepresentation $W'$ of $W$. We need to show that $W(a)(W'_{ta})\subset W'_{ha}$ for all $a\in Q_1$. Due to the construction of $W'$, it is more convenient to work with paths.  We split the paths in $Q$ into three major categories.

Case 1: The path starts at some $x_{i'}$. We consider three possibilities:
\begin{itemize}
\item Case (1a): The path ends at some $x_i$.
\item Case (1b): The path ends at some $y_{j}$.
\item Case (1c): The path ends at some $y$.
\end{itemize}
For case (1a), let $p:x_{i'}\rightarrow x_i$ be a path. We have seen in the proof of \Cref{subrep_preconstruction} that $W(p)(W_{x_{i'}})$, when placed in any $r$-th component of $W_{x_i}^{\theta_{+}(x_i)}$, is still in $U$. By property (1) of \Cref{subrep_preconstruction}, $W(p)(W'_{x_{i'}})\subset W'_{x_i}$. This finishes case (1a). Case (1b) follows from the description of $W'_{y_j}$ in \Cref{subrep_preconstruction}. Case (1c) follows from the construction of $W'_{y}$.

Case 2: The path starts at some $y_{j'}$. Again we consider three possibilities:
\begin{itemize}
\item Case (2a): The path ends at some $x_{i}$.
\item Case (2b): The path ends at some $y_{j}$.
\item Case (2c): The path ends at some $y$.
\end{itemize}
For case (2a), if there are no paths from any $x_{i'}$ into $y_{j}$, then $W'_{y_{j}}$ is the zero subspace, which is a trivial scenario. If there are paths from some $x_{i'}$ into $y_{j}$, the case is covered by case (1a). Case (2b) can be established using the fact that $V$ is an $\Endo(P_0)$-module like what we did for case (1a). Finally, case (2c) follows from the construction of $W'_{y}$ as well.

Case 3: The path starts at some $y'$. We consider:
\begin{itemize}
\item Case (3a): The path ends at some $x_{i}$
\item Case (3b): The path ends at some some $y_{j}$.
\item Case (3c): The path ends at some $y$.
\end{itemize}
If there are no paths from either $x_{i'}$ or $y_{j'}$ into $y$, then $W'_{y}$ must be the zero subspace, which is a trivial scenario. If there is a path from some $x_{i'}$ or some $y_{j'}$ into $y$, case (3a) is covered by case (1a) and case (2a). Similarly case (3b) is covered by case (1b) and case (2b). Finally, case (3c) is covered by the construction of $W'_{y}$.

We have shown that $\{W'_{x}\}_{x\in Q_0}$ forms a subrepresentation $W'$ of $W$. Next, note that $$\theta(W')=\sum_{i}(\dim W'_{x_i})\cdot \theta_{+}(x_i)-\sum_{j}(\dim W'_{y_j})\cdot \theta_{-}(y_j)=\dim U -\dim \mathcal{A}_{W,\theta}(U)$$ by \Cref{subrep_preconstruction}. By \Cref{rep_disc_bounded_by_mat_disc}, it must be the case that $\theta(W')=\disc(W,\theta)=\disc(\mathcal{A}_{W,\theta}).$
\end{proof}

We are now able to establish the algorithm of \Cref{discrepancy_rep_algorithm}, which returns the discrepancy, together with a witness for any representation of an acyclic quiver over an infinite field. 

\begin{corollary}\label{proposition3.5_hus}
Let $Q$ be an acylcic quiver and let $W$ be a representation as above. There is a deterministic algorithm that finds the discrepancy $\disc(W,\theta)$, together with a subrepresentation $W'$ so that $\theta(W')=\disc(W,\theta)$. If we set  $\Omega=\sum_{x\in Q_0}|\theta(x)|$, $K=\sum_{x\in Q_0}\dim W_{x}$, and $P$ as the number of paths in $Q$, then the algorithm has time complexity that is polynomial in $\Omega,K,P$. 
\end{corollary}

\begin{proof}
We propose the following three steps to find the discrepancy of $W$, together with a witnessing subrepresentation. 
As step 1, we compute a basis for the matrix space $$\mathcal{A}_{W,\theta}=\bigoplus_{i,j}\bigoplus_{q,r}\Span(A^{i,j,p}_{q,r}\mid p\in\mathcal{P}_{ij}),$$ where each $A^{i,j,p}_{q,r}$ is $W(p)$ placed in the $(q,r)$-block. If $p$ has length $l$, then $W(p)$ involves $l-1$ multiplications of matrices of sizes taken from a subset of $\{\dim W_x\}_{x\in Q_0}$. Iterating through paths in increasing lengths, computing $W(p)$ for all $p\in\mathcal{P}_{i,j}$ has time complexity polynomial in $K$, and linear in $P$. We then sort out linearly independent size $\dim W(y_j)\times \dim W(x_i)$ matrices among $\{W(p)\mid p\in \mathcal{P}_{i,j}\}$, which has time complexity polynomial in $K,P$. Finally ,we put these linearly independent matrices into appropriate $(q,r)$-blocks for all $q\in I^{+}_{j}$, $r\in I^{+}_i$, and for all $i,j$. There are $M\times M'$ blocks so the time complexity to compute a basis for $\mathcal{A}_{W,\theta}$ is polynomial in $\Omega,K,P$. 

As step 2, apply the algorithm of \Cref{thm_1.5_ISQ18} to obtain a basis of the minimal $c$-shrunk subspace $U$ of $\mathcal{A}_{W,\theta}$. The time complexity is polynomial in $N$, which in turn is bounded by a polynomial in $\Omega\cdot K$.

As step 3, we apply \Cref{from_mat_disc_to_rep_disc} to derive a basis for $W'_{x}$ for each $x\in Q_0$. A spanning set for each $W'_{x_i}$ is obtained by truncating suitable parts of the basis of $U$ we obtained earlier. We can then trim this spanning set down to a basis within time complexity polynomial in $\dim W_{x_i}$ and $N$, which is bounded by a polynomial in $\Omega,K$. Since there are at most $\Omega$ many $x_i'$s, the time complexity to obtain bases for all $W'_{x_i}$ is bounded by a polynomial in $\Omega,K$. For $W'_{y_j}$, we can first compute $\sum_{i}\sum_{p\in P_{i,j}}W(p)(W'_{x_i})$. In this case, since each $W(p)$ is computed earlier, a spanning set of $W(p)(W'_{x_i})$ is obtained by multiplying $W(p)$ with the basis matrix of $W'_{x_i}$. Hence a spanning set for $W'(y_j)$ can be obtained within time complexity polynomial in $K,P$. Since there are at most $P\cdot\sum_{i}\dim W'_{x_i}$ vectors in the spanning set for each $W'(y_j)$, deriving a basis for $W'_{y_{j}}$ has time complexity polynomial in $K,P$. Since the number of $y_{j}$'s is bounded by $\Omega$, we may infer that computing bases for all $W'_{y_j}$ has time complexity polynomial in $\Omega,K,P$. Similarly, computing bases for all $W'_{y}$ with $\theta(y)=0$ has time complexity polynomial in $\Omega,K,P$.

Since the three major steps are all within time complexity polynomial in $\Omega,K,P$. The corollary is proved.
\end{proof}

\end{document}